\documentclass{lmcs}
\pdfoutput=1

\usepackage{lastpage}
\lmcsdoi{15}{1}{15}
\lmcsheading{}{\pageref{LastPage}}{}{}%
{Dec.~19,~2017}{Feb.~22,~2019}{}

\usepackage{hyperref}
\usepackage{float}
\usepackage{amscd}
\usepackage{graphics}
\usepackage{cancel}
\usepackage{stmaryrd}
\usepackage{tikz}
\usetikzlibrary{matrix,arrows}
\usepackage{amsmath}
\usepackage{amsthm}
\usepackage{amssymb}
\usepackage{proof}
\usepackage{verbatim,cmll}
\usepackage{setspace}
\usepackage{lscape}
\usepackage{latexsym,relsize}
\usepackage{amscd}
\usepackage{multicol}
\usepackage{bussproofs}
\usepackage{tikz}
\usepackage[position=top]{subfig}
\usepackage{leqno}
\usepackage{marvosym}
\usetikzlibrary{arrows}
\usetikzlibrary{matrix}
\usetikzlibrary{patterns}
\usetikzlibrary{shapes}

\DeclareSymbolFont{extraup}{U}{zavm}{m}{n}
\DeclareMathSymbol{\varheart}{\mathalpha}{extraup}{86}
\DeclareMathSymbol{\vardiamond}{\mathalpha}{extraup}{87}
\DeclareMathSymbol{\vardiamond}{\mathalpha}{extraup}{87}

\newcommand\val[1]{{\lbrack\!\lbrack} {#1}{\rbrack\!\rbrack}}
\newcommand\valb[1]{{(\!\lbrack} {#1}{\rbrack\!)}}

\newcommand{\commment}[1]{}

\newcommand{\n}{\mathbf{n}}

\renewcommand{\phi}{\varphi}

\newcommand{\B}{\mathbb{B}}

\def\aol{\rule[0.5865ex]{1.38ex}{0.1ex}}
\def\pdra{\mbox{$\,>\mkern-8mu\raisebox{-0.065ex}{\aol}\,$}}
\renewcommand{\epsilon}{\varepsilon}

\newcommand{\nomi}{\mathbf{i}}
\newcommand{\nomj}{\mathbf{j}}

\newcommand{\cnomm}{\mathbf{m}}
\newcommand{\cnomn}{\mathbf{n}}

\newcommand{\bba}{\mathbb{A}}
\newcommand{\bbA}{\mathbb{A}}
\newcommand{\bbB}{\mathbb{B}}

\newcommand{\bbad}{\mathbb{A}^{\delta}}
\newcommand{\bbbd}{\mathbb{B}^{\delta}}
\newcommand{\kbbas}{K(\mathbb{A}^\delta)}
\newcommand{\obbas}{O(\mathbb{A}^\delta)}

\newcommand{\marginnote}[1]{\marginpar{\raggedright\tiny{#1}}}

\theoremstyle{plain}

\newcommand*\circled[1]{\tikz[baseline=(char.base)]{
		\node[shape=circle ,draw, minimum size=3.5mm, inner sep=0pt] (char) {#1};}}

\tikzset{
	treenode/.style = {align=center, inner sep=0pt, text centered},
	Ske/.style = {treenode, ellipse, double, draw=black,
		minimum width=6pt, thick},
	PIA/.style = {treenode, ellipse, black, draw=black,
		minimum width=6pt},
	Crit/.style = {treenode, rectangle, draw=black,
		minimum width=0.5em, minimum height=0.5em}
}

\theoremstyle{plain}
\newtheorem{theorem}{Theorem}[section]

\newtheorem{example}[theorem]{Example}
\newtheorem{lemma}[thm]{Lemma}

\theoremstyle{definition}

\newtheorem{definition}[thm]{Definition}
\newtheorem{remark}[thm]{Remark}

\begin{document}

\title[Sahlqvist via Translation]{Sahlqvist via Translation}

\author[Conradie, Palmigiano and Zhao]{Willem Conradie}	
\address{School of Mathematics, University of the Witwatersrand, Johannesburg, South Africa}	
\email{willem.conradie@wits.ac.za}  
\thanks{The research of the first author was partially supported by the National Research Foundation of South Africa, Grant number 81309.}	

\author[]{Alessandra Palmigiano}	
\address{Faculty of Technology, Policy and Management, Delft University of Technology, the Netherlands, and Department of Pure and Applied Mathematics, University of Johannesburg, South Africa}	
\email{A.Palmigiano@tudelft.nl}  
\thanks{The research of the second and third authors has been made possible by the NWO Vidi grant 016.138.314, by the NWO Aspasia grant 015.008.054, and by a Delft Technology Fellowship awarded in 2013.}	

\author[]{Zhiguang Zhao}	
\address{Faculty of Technology, Policy and Management, Delft University of Technology, the Netherlands}	
\email{zhaozhiguang23@gmail.com}  





\begin{abstract}
  \noindent In recent years, {\em unified correspondence} has been developed as a generalized Sahlqvist theory which applies uniformly to all signatures  of normal and regular (distributive) lattice expansions. A fundamental tool for attaining this level of generality and uniformity is a principled way, based on order theory, to define the Sahlqvist and inductive formulas and inequalities in every such signature. This  definition covers in particular all  (bi-)intuitionistic modal logics. The theory of these logics has been intensively studied over the past seventy years in connection with classical polyadic modal logics, using versions of G\"odel-McKinsey-Tarski translations, suitably defined in each signature, as main tools. In view of this state-of-the-art, it is natural to ask (1) whether a general perspective on G\"odel-McKinsey-Tarski translations can be attained, also based on order-theoretic principles like those underlying the general definition of Sahlqvist and inductive formulas and inequalities, which accounts for the known G\"odel-McKinsey-Tarski translations and applies uniformly to all signatures of normal (distributive) lattice expansions; (2) whether this general perspective can be used to transfer  correspondence and canonicity theorems for Sahlqvist and inductive formulas and inequalities in all signatures described above under G\"odel-McKinsey-Tarski translations.
  
  In the present paper, we set out to answer these questions. We answer (1) in the affirmative; as to (2), we prove the transfer of the correspondence theorem  for inductive inequalities of  arbitrary signatures of normal distributive lattice expansions. We also prove the transfer of  canonicity  for inductive inequalities, but only restricted to arbitrary normal modal expansions of {\em bi-intuitionistic logic}. We also analyze  the difficulties involved in obtaining the transfer of canonicity outside this setting, and indicate a route to extend the transfer of canonicity to all signatures of normal distributive lattice expansions.
\end{abstract}

\maketitle

\section{Introduction} 
Sahlqvist theory has a long history in normal modal logic, going back to \cite{Sa75}. The Sahlqvist theorem in \cite{Sa75} gives a syntactic definition of a class of modal formulas, the {\em Sahlqvist class}, each member of which defines an elementary (i.e.\ first-order definable) class of frames and is canonical.

Over the years, many extensions, variations and analogues of this result have appeared, including alternative proofs in e.g.\ \cite{SaVa89}, generalizations to arbitrary modal signatures \cite{dRVe95}, variations of the correspondence language \cite{OhSc97,Benthem06}, Sahlqvist-type results for hybrid logics  \cite{tCMaVi06},  various substructural logics \cite{Kurtonina,DunnGP05,Ge06}, mu-calculus  \cite{BeHovB12}, and enlargements of the Sahlqvist class to e.g.\ the {\em inductive} formulas of \cite{GorankoV06}, to mention but a few.

Recently, a uniform and modular theory has emerged, called {\em unified correspondence} \cite{CoGhPa14}, which subsumes the above results and extends them to logics with a {\em non-classical} propositional base. It is built on duality-theoretic insights \cite{ConPalSou12} and uniformly exports the state-of-the-art in Sahlqvist theory from normal modal logic to a wide range of logics which include, among others, intuitionistic and distributive and general (non-distributive) lattice-based (modal) logics \cite{ConPal12, ConPal13},  non-normal (regular) modal logics based on distributive lattices of arbitrary modal signature \cite{PaSoZh15b}, hybrid logics \cite{CoRo14}, many valued logics, \cite{manyval} and bi-intuitionistic and lattice-based modal mu-calculus \cite{CoCr14,CoFoPaSo15,long-version}.

The breadth of this work has stimulated many and varied applications. Some are closely related to the core concerns of the theory itself, such as understanding the relationship between different methodologies for obtaining canonicity results \cite{PaSoZh15a,CPZ:constructive}, or the exploration of the limits of applicability of the theory \cite{zhao2017Hausdorff, zhao2016possibility} or of the phenomenon of pseudocorrespondence \cite{CoPaSoZh15}.  Other, possibly surprising applications include the dual characterizations of classes of finite lattices \cite{FrPaSa14}, the identification of the syntactic shape of axioms which can be translated into structural rules of a proper display calculus \cite{GrMaPaTzZh15} and of internal Gentzen calculi for the logics of strict implication \cite{MaZh15}, and the epistemic interpretation of lattice-based modal logic in terms of categorization theory in management science \cite{categorization, TarkPaper}. The approach underlying these results relies only on  the order-theoretic properties of the algebraic interpretations of logical connectives, abstracting away from specific logical signatures.

Featuring prominently among the logics targeted by  these developments are  (bi-)intuitionistic logic and their all modal expansions. 
The theory of these logics has been intensively studied over the past seventy years using the  G\"{o}del-McKinsey-Tarski translation \cite{Godel:Interp, McKinseyTarski:Translation}, henceforth simply the GMT translation, as a key tool. 
Specifically, since the 1940s and up to the present day, versions of the GMT translation have been used for transferring and reflecting results between classical and intuitionistic logics and their extensions and expansions (see e.g.\ \cite{mckinsey1948some, blok1976varieties, esakia1976modal, fischerservi1977modal, ChagrovZ91, WoZa97, WoZa98, wolter2014blok, vBBeHo16}. More on this in Section \ref{ssec:brief history}). 

In view of this state-of-the-art, it is natural to ask (1) whether a general perspective on G\"odel-McKinsey-Tarski translations can be attained, also based on  order-theoretic principles like those underlying the general definition of Sahlqvist and inductive formulas and inequalities, which accounts for the known G\"odel-McKinsey-Tarski translations and applies uniformly to all signatures of normal (distributive) lattice expansions; (2) whether this general perspective can be used to transfer  correspondence and canonicity theorems for Sahlqvist and inductive formulas and inequalities in the signatures described above under G\"odel-McKinsey-Tarski translations.
%
In the present paper, we set out to answer these questions.

Notice that in general, GMT translations do not preserve the Sahlqvist shape.
For instance, 
the original GMT translation 
transforms the Sahlqvist inequality $\Box\Diamond p\leq \Diamond p$ into ${\Box\Diamond\Box_G p}\leq \Diamond \Box_G p$,  which is not Sahlqvist, and in fact does not even have a first-order correspondent \cite{vanBenthem:Reduction:Principles}. Any translation which `boxes' propositional variables would suffer from this problem (for further discussion  see \cite[Section 36.9]{CoGhPa14}).
However, {\em some} GMT translations preserve the shape of {\em some} Sahlqvist and inductive formulas or inequalities.
This has been exploited by Gehrke, Nagahashi and Venema in \cite{GeNaVe05} to obtain the correspondence part of their Sahlqvist theorem for Distributive Modal Logic. The need to establish a suitable match between GMT translations and Sahlqvist and inductive formulas or inequalities in each signature gives us a concrete reason to investigate GMT translations as a class. 
%

The starting point of our analysis, and first contribution of the present paper, is an order-theoretic  generalization of the main semantic property of the original GMT translation. We show that  this generalization provides a unifying pattern, instantiated in the concrete GMT translations in each setting of interest to the present paper. As an application of this generalization, we prove two transfer results, which are the main contributions of the present paper: the transfer of {\em generalized Sahlqvist correspondence} from multi-modal (classical) modal logic to logics of arbitrary normal distributive lattice expansions and, in a more restricted setting, the transfer of {\em generalized Sahlqvist canonicity} from multi-modal (classical) modal logic  to logics of arbitrary normal {\em bi-Heyting algebra} expansions. The transfer of correspondence extends \cite[Theorem 3.7]{GeNaVe05} both as regards the setting (from Distributive Modal Logic to arbitrary normal DLE-logics) and the scope (from Sahlqvist to inductive inequalities). The transfer of canonicity is entirely novel also in its formulation, in the sense that it targets specific, syntactically defined classes. Finally, we  analyze the difficulties in extending the transfer of canonicity to  normal DLE-logics. Thanks to this analysis,  we identify a route towards this result, which -- however -- calls for a much higher level of technical sophistication than required by the usual route.

The paper is structured as follows. Section \ref{sec:preliminaries} collects preliminaries on  logics of normal DLEs and their algebraic and relational semantics. In Section \ref{semantic environment} we discuss  GMT translations in various DLE settings in the literature  and the semantic underpinnings of the GMT translation for intuitionistic logic. In Section \ref{ssec:template} we introduce a general template which accounts for the main semantic property of GMT translations---their being full and faithful---in the setting of ordered algebras of arbitrary signatures. In Section \ref{sec:instantiations}, we show that the GMT translations of interest instantiate this template. This sets the stage for Sections \ref{sec:Corresp:Via:Trans} and \ref{sec:canonicity} where we present our transfer results. 
We conclude in Section \ref{Sec:Conclusions}.

\section{Preliminaries on normal DLEs and their logics}\label{sec:preliminaries}

In this section we collect some basic background on logics of normal distributive lattice expansions (DLEs). All the logics which we will consider in this paper are particular instances of DLE-logics.

\subsection{Language and axiomatization of basic DLE-logics}\label{subset:language:algsemantics}
Our base language is an unspecified but fixed language $\mathcal{L}_\mathrm{DLE}$, to be interpreted over distributive lattice expansions of compatible similarity type.
We will make heavy use of the following auxiliary definition: an {\em order-type} over $n\in \mathbb{N}$\footnote{Throughout the paper, order-types will be typically associated with arrays of variables $\vec p: = (p_1,\ldots, p_n)$. When the order of the variables in $\vec p$ is not specified, we will sometimes abuse notation and write $\varepsilon(p) = 1$ or $\varepsilon(p) = \partial$.} is an $n$-tuple $\epsilon\in \{1, \partial\}^n$. For every order type $\epsilon$, we denote its {\em opposite} order type by $\epsilon^\partial$, that is, $\epsilon^\partial_i = 1$ iff $\epsilon_i=\partial$ for every $1 \leq i \leq n$. For any lattice $\bba$, we let $\bba^1: = \bba$ and $\bba^\partial$ be the dual lattice, that is, the lattice associated with the converse partial order of $\bba$. For any order type $\varepsilon$, we let $\bba^\varepsilon: = \Pi_{i = 1}^n \bba^{\varepsilon_i}$.

The language $\mathcal{L}_\mathrm{DLE}(\mathcal{F}, \mathcal{G})$ (from now on abbreviated as $\mathcal{L}_\mathrm{DLE}$) takes as parameters: 1) a denumerable set $\mathsf{PROP}$ of proposition letters, elements of which are denoted $p,q,r$, possibly with indexes; 2) disjoint sets of connectives $\mathcal{F}$ and  $\mathcal{G}$. 
Each $f\in \mathcal{F}$ and $g\in \mathcal{G}$ has arity $n_f\in \mathbb{N}$ (resp.\ $n_g\in \mathbb{N}$) and is associated with some order-type $\varepsilon_f$ over $n_f$ (resp.\ $\varepsilon_g$ over $n_g$).\footnote{Unary $f$ (resp.\ $g$) will be sometimes denoted as $\Diamond$ (resp.\ $\Box$) if the order-type is 1, and $\lhd$ (resp.\ $\rhd$) if the order-type is $\partial$.} 
The terms (formulas) of $\mathcal{L}_\mathrm{DLE}$ are defined recursively as follows:
\[
\phi ::= p \mid \bot \mid \top \mid \phi \wedge \phi \mid \phi \vee \phi \mid f(\overline{\phi}) \mid g(\overline{\phi})
\]
where $p \in \mathsf{PROP}$, $f \in \mathcal{F}$, $g \in \mathcal{G}$. Terms in $\mathcal{L}_\mathrm{DLE}$ will be denoted either by $s,t$, or by lowercase Greek letters such as $\varphi, \psi, \gamma$ etc. 

\begin{definition}
	\label{def:DLE:logic:general}
	For any language $\mathcal{L}_\mathrm{DLE} = \mathcal{L}_\mathrm{DLE}(\mathcal{F}, \mathcal{G})$, the {\em basic}, or {\em minimal} $\mathcal{L}_\mathrm{DLE}$-{\em logic} is a set of sequents $\phi\vdash\psi$, with $\phi,\psi\in\mathcal{L}_\mathrm{LE}$, which contains the following axioms:
	\begin{itemize}
		\item Sequents for  lattice operations:
		\begin{align*}
		&p\vdash p, && \bot\vdash p, && p\vdash \top, & & p\wedge (q\vee r)\vdash (p\wedge q)\vee (p\wedge r), &\\
		&p\vdash p\vee q, && q\vdash p\vee q, && p\wedge q\vdash p, && p\wedge q\vdash q, &
		\end{align*}
		\item Sequents for $f\in \mathcal{F}$ and $g\in \mathcal{G}$:
		\begin{align*}
		& f(p_1,\ldots, \bot,\ldots,p_{n_f}) \vdash \bot,~\mathrm{for}~ \varepsilon_f(i) = 1,\\
		& f(p_1,\ldots, \top,\ldots,p_{n_f}) \vdash \bot,~\mathrm{for}~ \varepsilon_f(i) = \partial,\\
		&\top\vdash g(p_1,\ldots, \top,\ldots,p_{n_g}),~\mathrm{for}~ \varepsilon_g(i) = 1,\\
		&\top\vdash g(p_1,\ldots, \bot,\ldots,p_{n_g}),~\mathrm{for}~ \varepsilon_g(i) = \partial,\\
		&f(p_1,\ldots, p\vee q,\ldots,p_{n_f}) \vdash f(p_1,\ldots, p,\ldots,p_{n_f})\vee f(p_1,\ldots, q,\ldots,p_{n_f}),~\mathrm{for}~ \varepsilon_f(i) = 1,\\
		&f(p_1,\ldots, p\wedge q,\ldots,p_{n_f}) \vdash f(p_1,\ldots, p,\ldots,p_{n_f})\vee f(p_1,\ldots, q,\ldots,p_{n_f}),~\mathrm{for}~ \varepsilon_f(i) = \partial,\\
		& g(p_1,\ldots, p,\ldots,p_{n_g})\wedge g(p_1,\ldots, q,\ldots,p_{n_g})\vdash g(p_1,\ldots, p\wedge q,\ldots,p_{n_g}),~\mathrm{for}~ \varepsilon_g(i) = 1,\\
		& g(p_1,\ldots, p,\ldots,p_{n_g})\wedge g(p_1,\ldots, q,\ldots,p_{n_g})\vdash g(p_1,\ldots, p\vee q,\ldots,p_{n_g}),~\mathrm{for}~ \varepsilon_g(i) = \partial,
		\end{align*}
	\end{itemize}
	and is closed under the following inference rules:
	\begin{displaymath}
	\frac{\phi\vdash \chi\quad \chi\vdash \psi}{\phi\vdash \psi}
	\quad
	\frac{\phi\vdash \psi}{\phi(\chi/p)\vdash\psi(\chi/p)}
	\quad
	\frac{\chi\vdash\phi\quad \chi\vdash\psi}{\chi\vdash \phi\wedge\psi}
	\quad
	\frac{\phi\vdash\chi\quad \psi\vdash\chi}{\phi\vee\psi\vdash\chi}
	\end{displaymath}
	\begin{displaymath}
	\frac{\phi\vdash\psi}{f(p_1,\ldots,\phi,\ldots,p_n)\vdash f(p_1,\ldots,\psi,\ldots,p_n)}{~(\varepsilon_f(i) = 1)}
	\end{displaymath}
	\begin{displaymath}
	\frac{\phi\vdash\psi}{f(p_1,\ldots,\psi,\ldots,p_n)\vdash f(p_1,\ldots,\phi,\ldots,p_n)}{~(\varepsilon_f(i) = \partial)}
	\end{displaymath}
	\begin{displaymath}
	\frac{\phi\vdash\psi}{g(p_1,\ldots,\phi,\ldots,p_n)\vdash g(p_1,\ldots,\psi,\ldots,p_n)}{~(\varepsilon_g(i) = 1)}
	\end{displaymath}
	\begin{displaymath}
	\frac{\phi\vdash\psi}{g(p_1,\ldots,\psi,\ldots,p_n)\vdash g(p_1,\ldots,\phi,\ldots,p_n)}{~(\varepsilon_g(i) = \partial)}.
	\end{displaymath}
  \\[2mm]
	The minimal DLE-logic is denoted by $\mathbf{L}_\mathrm{DLE}$. For any DLE-language $\mathcal{L}_{\mathrm{DLE}}$, by an {\em $\mathrm{DLE}$-logic} we understand any axiomatic extension of the basic $\mathcal{L}_{\mathrm{DLE}}$-logic in $\mathcal{L}_{\mathrm{DLE}}$.
\end{definition}

\begin{example}
	\label{ex:various DLE-languages}
	The DLE setting is extremely general, covering a wide spectrum of non-classical logics. Here we give a few examples, showing how various languages are obtained as specific instantiations of $\mathcal{F}$ and $\mathcal{G}$. The associated logics extended the basic DLE logics corresponding to these instantiations.
	
	The formulas of intuitionistic logic are obtained by instantiating $\mathcal{F}: = \varnothing$   and   $\mathcal{G}: = \{\rightarrow\}$ with $n_\rightarrow = 2$, and $\varepsilon_\rightarrow = (\partial, 1)$. The formulas of bi-intuitionistic logic (cf.~\cite{rauszer1974semi})  are obtained by instantiating $\mathcal{F}: = \{\pdra \}$ with $n_{\pdra}  = 2$ and $\varepsilon_{\pdra} = (\partial, 1)$,  and   $\mathcal{G}: = \{\rightarrow\}$ with $n_\rightarrow = 2$ and $\varepsilon_\rightarrow = (\partial, 1)$. The formulas of Fischer Servi's intuitionistic modal logic (cf.~\cite{fischerservi1977modal}), Prior's MIPC (cf.~\cite{prior1955time}), and G.~Bezhanishvili's intuitionistic modal logic with universal modalities (cf.~\cite{bezhanishvili2009universal}) are obtained by instantiating $\mathcal{F}: = \{\Diamond\}$ with $n_\Diamond  = 1$ and $\varepsilon_\Diamond = 1$   and   $\mathcal{G}: = \{\rightarrow, \Box\}$ with $n_\rightarrow = 2$, and $\varepsilon_\rightarrow = (\partial, 1)$, and $n_\Box = 1$, and $\varepsilon_\Box = 1$.
	The formulas of Wolter's bi-intuitionistic modal logic (cf.~\cite{wolter1998CoImplication}) are obtained by instantiating $\mathcal{F}: = \{\pdra, \Diamond\}$ with $n_\Diamond  = 1$ and $\varepsilon_\Diamond = 1$  and $n_{\pdra}  = 2$ and $\varepsilon_{\pdra} = (\partial, 1)$,  and   $\mathcal{G}: = \{\rightarrow, \Box\}$ with $n_\rightarrow = 2$, and $\varepsilon_\rightarrow = (\partial, 1)$, and $n_\Box = 1$, and $\varepsilon_\Box = 1$.  The formulas of Dunn's \emph{positive modal logic} (cf.\ \cite{dunn1995positive}) are obtained by instantiating $\mathcal{F}: = \{\Diamond\}$ with $n_\Diamond = 1$, $\varepsilon_\Diamond = 1$  and  $\mathcal{G}: = \{\Box \}$ with $n_\Box = 1$ and $\varepsilon_\Box = 1$. The language of Gehrke, Nagahashi and Venema's \emph{distributive modal logic} (cf.\ \cite{GeNaVe05}) is an expansion of positive modal logic and is obtained by adding the connectives ${\lhd}$ and ${\rhd}$ to $\mathcal{F}$ and $\mathcal{G}$, respectively, with $n_\lhd = n_\rhd = 1$ and $\varepsilon_\lhd = \varepsilon_\rhd = \partial$.
\end{example}

\subsection{Algebraic and relational semantics for basic DLE-logics}\label{ssec: DLEs}
The following definition captures the algebraic setting of the present paper: 

\begin{definition}
	\label{def:DLE}
	For any tuple $(\mathcal{F}, \mathcal{G})$ of disjoint sets of function symbols as above, a {\em distributive lattice expansion} (abbreviated as DLE) is a tuple $\bba = (L, \mathcal{F}^\bbA, \mathcal{G}^\bbA)$ such that $L$ is a bounded distributive lattice, $\mathcal{F}^\bbA = \{f^\bbA\mid f\in \mathcal{F}\}$ and $\mathcal{G}^\bbA = \{g^\bbA\mid g\in \mathcal{G}\}$, such that every $f^\bbA\in\mathcal{F}^\bbA$ (resp.\ $g^\bbA\in\mathcal{G}^\bbA$) is an $n_f$-ary (resp.\ $n_g$-ary) operation on $\bbA$. A DLE is {\em normal} if every $f^\bbA\in\mathcal{F}^\bbA$ (resp.\ $g^\bbA\in\mathcal{G}^\bbA$) preserves finite (hence also empty) joins (resp.\ meets) in each coordinate with $\epsilon_f(i)=1$ (resp.\ $\epsilon_g(i)=1$) and reverses finite (hence also empty) meets (resp.\ joins) in each coordinate with $\epsilon_f(i)=\partial$ (resp.\ $\epsilon_g(i)=\partial$).\footnote{\label{footnote:DLE vs DLO} Normal DLEs are sometimes referred to as {\em distributive lattices with operators} (DLOs). This terminology derives from the setting of Boolean algebras with operators, in which operators are understood as operations which preserve finite (hence also empty) joins in each coordinate. Thanks to the Boolean negation, operators are typically taken as primitive connectives, and all the other modal operations are reduced to these. However, this terminology is somewhat ambiguous in the lattice setting, in which primitive operations are typically maps which are operators if seen as $\bbA^\epsilon\to \bbA^\eta$ for some order-type $\epsilon$ on $n$ and some order-type $\eta\in \{1, \partial\}$. Rather than speaking of lattices with $(\varepsilon, \eta)$-operators, we then speak of normal DLEs. This terminology is also used in other papers developing Sahlqvist-type results at a  level of generality comparable to that of the present paper, e.g.~\cite{GrMaPaTzZh15, long-version}.} Let $\mathbb{DLE}$ be the class of normal DLEs. Sometimes we will refer to certain DLEs as $\mathcal{L}_\mathrm{DLE}$-algebras when we wish to emphasize that these algebras have a compatible signature with the logical language we have fixed.
	A distributive lattice is {\em perfect} if it is complete, completely distributive and completely join-generated by  its completely join-prime elements. Equivalently, a distributive lattice is perfect iff it is isomorphic to the lattice of up-sets of some poset. A normal DLE is {\em perfect} if its lattice-reduct is a perfect distributive lattice, and each $f$-operation (resp.\ $g$-operation) is completely join-preserving (resp.\ meet-preserving) in the coordinates $i$ such that $\epsilon_f(i) = 1$ (resp.~$\epsilon_g(i) = 1$) and completely meet-reversing (resp.\ join-reversing) in the coordinates $i$ such that $\epsilon_f(i) = \partial$ (resp.~$\epsilon_g(i) = \partial$). The {\em canonical extension}  of a  normal DLE $\mathbb{A} = (L, \mathcal{F}, \mathcal{G})$
	is  the perfect normal DLE $\mathbb{A}^\delta: = (L^\delta, \mathcal{F}^\sigma, \mathcal{G}^\pi)$, where $L^\delta$  is the canonical extension of  $L$,\footnote{The \emph{canonical extension} of a bounded  lattice $L$ is a complete  lattice $L^\delta$ containing $L$ as a sublattice, such that:
		\begin{enumerate}
			\item \emph{(denseness)} every element of $L^\delta$ is both the join of meets and the meet of joins of elements from $L$;
			\item \emph{(compactness)} for all $S,T \subseteq L$, if $\bigwedge S \leq \bigvee T$ in $L^\delta$, then $\bigwedge F \leq \bigvee G$ for some finite sets $F \subseteq S$ and $G\subseteq T$.
		\end{enumerate}
	}
	and $\mathcal{F}^\sigma: = \{f^\sigma\mid f\in \mathcal{F}\}$ and  $\mathcal{G}^\pi: = \{g^\pi\mid g\in \mathcal{G}\}$.\footnote{An element $k \in L^\delta$ (resp.~$o\in L^\delta$) is \emph{closed} (resp.\ \emph{open}) if is the meet (resp.\ join) of some subset of $L$. We let $K(L^\delta)$ (resp.~$O(L^\delta)$) denote the set of the closed (resp.~open) elements of $L^\delta$. For every unary, order-preserving map $h : L \to M$ between bounded lattices, the $\sigma$-{\em extension} of $h$ is defined firstly by declaring, for every $k\in K(L^\delta)$,
		$$h^\sigma(k):= \bigwedge\{ h(a)\mid a\in L\mbox{ and } k\leq a\},$$ and then, for every $u\in L^\delta$,
		$$h^\sigma(u):= \bigvee\{ h^\sigma(k)\mid k\in K(L^\delta)\mbox{ and } k\leq u\}.$$
		The $\pi$-{\em extension} of $f$ is defined firstly by declaring, for every $o\in O(L^\delta)$,
		$$h^\pi(o):= \bigvee\{ h(a)\mid a\in L\mbox{ and } a\leq o\},$$ and then, for every $u\in L^\delta$,
		$$h^\pi(u):= \bigwedge\{ h^\pi(o)\mid o\in O(L^\delta)\mbox{ and } u\leq o\}.$$
		The definitions above apply also to operations of any finite arity and order-type. Indeed,
		taking  order-duals interchanges closed and open elements:
		$K({(L^\delta)}^\partial) = O(L^\delta)$ and $O({(L^\delta)}^\partial) = K(L^\delta)$;  similarly, $K({(L^n)}^\delta) =K(L^\delta)^n$, and $O({(L^n)}^\delta) =O(L^\delta)^n$. Hence,  $K({(L^\delta)}^\epsilon) =\prod_i K(L^\delta)^{\epsilon(i)}$ and $O({(L^\delta)}^\epsilon) =\prod_i O(L^\delta)^{\epsilon(i)}$ for every lattice $L$ and every order-type $\epsilon$ over any $n\in \mathbb{N}$, where
		\begin{center}
			\begin{tabular}{cc}
				$K(L^\delta)^{\epsilon(i)}: =\begin{cases}
				K(L^\delta) & \mbox{if } \epsilon(i) = 1\\
				O(L^\delta) & \mbox{if } \epsilon(i) = \partial\\
				\end{cases}
				$ &
				$O(L^\delta)^{\epsilon(i)}: =\begin{cases}
				O(L^\delta) & \mbox{if } \epsilon(i) = 1\\
				K(L^\delta) & \mbox{if } \epsilon(i) = \partial.\\
				\end{cases}
				$\\
			\end{tabular}
		\end{center}
		From this it follows that
		${(L^\partial)}^\delta$ can be  identified with ${(L^\delta)}^\partial$,  ${(L^n)}^\delta$ with ${(L^\delta)}^n$, and
		${(L^\epsilon)}^\delta$ with ${(L^\delta)}^\epsilon$ for any order type $\epsilon$ over $n$, where $L^\epsilon: = \prod_{i = 1}^n L^{\epsilon(i)}$.
		These identifications make it possible to obtain the definition of $\sigma$-and $\pi$-extensions of $\epsilon$-monotone operations of any arity $n$ and order-type $\epsilon$ over $n$ by instantiating the corresponding definitions given above for monotone and unary functions.
	}
	Canonical extensions of  Heyting algebras, Brouwerian algebras and bi-Heyting algebras are defined by instantiating the definition above in the corresponding  signatures. The canonical extension of any Heyting (resp.~Brouwerian, bi Heyting) algebra is a (perfect) Heyting (resp.~Brouwerian, bi-Heyting) algebra.
\end{definition}

In the present paper we also find it convenient to talk of normal Boolean algebra expansions (BAEs) (respectively, normal Heyting algebra expansions (HAEs), normal bi-Heyting algebra expansions (bHAEs)) which are structures defined as in Definition \ref{def:DLE}, but replacing the distributive lattice $L$ with a Boolean algebra (respectively, Heyting algebra, bi-Heyting algebra). The logics corresponding to these classes will be collectively referred to as normal BAE-logics, normal HAE-logics, normal bHAE-logics. In what follows we will typically drop the adjective `normal'.

In the remainder of the paper, we will abuse notation and write e.g.\ $f$ for $f^\bbA$ when this causes no confusion.
Normal DLEs constitute the main semantic environment of the present paper. Henceforth, since every DLE is assumed to be normal, the adjective will be typically dropped.
The class of all DLEs is equational, and can be axiomatized by the usual distributive lattice identities and the following equations for any $f\in \mathcal{F}$ (resp.\ $g\in \mathcal{G}$) and $1\leq i\leq n_f$ (resp.\ for each $1\leq j\leq n_g$):
\begin{itemize}
	\item if $\varepsilon_f(i) = 1$, then $f(p_1,\ldots, p\vee q,\ldots,p_{n_f}) = f(p_1,\ldots, p,\ldots,p_{n_f})\vee f(p_1,\ldots, q,\ldots,p_{n_f})$; moreover if $f\in \mathcal{F}_n$, then $f(p_1,\ldots, \bot,\ldots,p_{n_f}) = \bot$,
	\item if $\varepsilon_f(i) = \partial$, then $f(p_1,\ldots, p\wedge q,\ldots,p_{n_f}) = f(p_1,\ldots, p,\ldots,p_{n_f})\vee f(p_1,\ldots, q,\ldots,p_{n_f})$; moreover if $f\in \mathcal{F}_n$, then  $f(p_1,\ldots, \top,\ldots,p_{n_f}) = \bot$,
	\item if $\varepsilon_g(j) = 1$, then $g(p_1,\ldots, p\wedge q,\ldots,p_{n_g}) = g(p_1,\ldots, p,\ldots,p_{n_g})\wedge g(p_1,\ldots, q,\ldots,p_{n_g})$; moreover if $g\in \mathcal{G}_n$, then  $g(p_1,\ldots, \top,\ldots,p_{n_g}) = \top$,
	\item if $\varepsilon_g(j) = \partial$, then $g(p_1,\ldots, p\vee q,\ldots,p_{n_g}) = g(p_1,\ldots, p,\ldots,p_{n_g})\wedge g(p_1,\ldots, q,\ldots,p_{n_g})$; moreover if $g\in \mathcal{G}_n$, then  $g(p_1,\ldots, \bot,\ldots,p_{n_g}) = \top$.
\end{itemize}
Each language $\mathcal{L}_\mathrm{DLE}$ is interpreted in the appropriate class of DLEs. In particular, for every DLE $\bbA$, each operation $f^\bbA\in \mathcal{F}^\bbA$ (resp.\ $g^\bba\in \mathcal{G}^\bbA$) is finitely join-preserving (resp.\ meet-preserving) in each coordinate when regarded as a map $f^\bbA: \bba^{\varepsilon_f}\to \bba$ (resp.\ $g^\bba: \bba^{\varepsilon_g}\to \bba$). 

For every DLE $\bbA$, the symbol $\vdash$ is interpreted as the lattice order $\leq$. A sequent $\phi\vdash\psi$ is valid in $\bba$ if $h(\phi)\leq h(\psi)$ for every homomorphism $h$ from the $\mathcal{L}_\mathrm{DLE}$-algebra of formulas over $\mathsf{PROP}$ to $\bba$. The notation $\mathbb{DLE}\models\phi\vdash\psi$ indicates that $\phi\vdash\psi$ is valid in every DLE. Then, by means of a routine Lindenbaum-Tarski construction, it can be shown that the minimal DLE-logic $\mathbf{L}_\mathrm{DLE}$ is sound and complete with respect to its corresponding class of algebras $\mathbb{DLE}$, i.e.\ that any sequent $\phi\vdash\psi$ is provable in $\mathbf{L}_\mathrm{DLE}$ iff $\mathbb{DLE}\models\phi\vdash\psi$. 

\begin{definition}
	\label{def:DLE frame}
	An $\mathcal{L}_{\mathrm{DLE}}$-{\em frame} is a tuple $\mathbb{F} = (\mathbb{X}, \mathcal{R}_{\mathcal{F}}, \mathcal{R}_{\mathcal{G}})$ such that $\mathbb{X} = (W, \leq)$ is a (nonempty) poset, $\mathcal{R}_{\mathcal{F}} = \{R_f\mid f\in \mathcal{F}\}$, and $\mathcal{R}_{\mathcal{G}} = \{R_g\mid g\in \mathcal{G}\}$ such that  for each $f\in \mathcal{F}$, the symbol $R_f$  denotes an $(n_f+1)$-ary  relation on $W$ such that
	for all $\overline{w}, \overline{v}\in  \mathbb{X}^{\eta_f}$,
	\begin{equation}\label{eq:compatibility Rf}
	\mbox{if } R_f(\overline{w})\ \mbox{ and }\ \overline{w}\leq^{\eta_f} \overline{v},\ \mbox{ then }\  R_f(\overline{v}),\end{equation}
	where $\eta_f$ is the order-type on $n_f+1$ defined as follows: $\eta_f(1) = 1$ and $\eta_f(i+1) = \epsilon^\partial_f(i)$ for each $1\leq i\leq n_f$.
	
	Likewise, for each $g\in \mathcal{G}$, the symbol  $R_g$ denotes an $(n_g+1)$-ary relation on $W$ such that
	for all $\overline{w}, \overline{v}\in  \mathbb{X}^{\eta_g}$,
	\begin{equation}\label{eq:compatibility Rg}
	\mbox{if } R_g(\overline{w})\   \mbox{ and }\  \overline{w}\geq^{\eta_g} \overline{v},\ \mbox{ then }\ R_g(\overline{v}),\end{equation}
	where $\eta_g$ is the order-type on $n_g+1$ defined as follows: $\eta_g(1) = 1$ and $\eta_g(i+1) = \epsilon_g^\partial(i)$ for each $1\leq i\leq n_g$.
	
	An $\mathcal{L}_{\mathrm{DLE}}$-{\em model} is a tuple $\mathbb{M} = (\mathbb{F}, V)$ such that $\mathbb{F}$ is an $\mathcal{L}_{\mathrm{DLE}}$-frame, and $V:\mathsf{Prop}\to \mathcal{P}^{\uparrow}(W)$ is a persistent valuation.
\end{definition}

The defining clauses for the interpretation of each $f\in \mathcal{F}$ and $g\in \mathcal{G}$ on $\mathcal{L}_{\mathrm{DLE}}$-models are given as follows: 
\begin{center}
	\begin{tabular}{l l l}
		$\mathbb{M},w \Vdash f(\overline{\phi})$ & $\quad\mbox{ iff }\quad $ & $ \mbox{there exists some } \overline{v} \in W^{n_f} \mbox{ s.t. } R_f(w, \overline{v}) $ \\
    & &$\mbox{and } \mathbb{M}, v_i \Vdash^{\epsilon_f(i)} \phi_i \mbox{ for each } 1\leq i\leq n_f,$\\
		
		$\mathbb{M},w \Vdash g(\overline{\phi})$ & $\quad\mbox{ iff }\quad $ &$ \mbox{for any } \overline{v} \in W^{n_g}, \mbox{ if } R_g(w, \overline{v}) \mbox{ then } \mathbb{M}, v_i \Vdash^{\epsilon_g(i)} \phi_i $
                                                                           \\
   & & $\mbox{for some } 1\leq i\leq n_g,$\\
	\end{tabular}
\end{center}
where $\Vdash^1$ is $\Vdash$ and $\Vdash^\partial$ is $\nVdash$.

\subsection{Sahlqvist and Inductive $\mathcal{L}_\mathrm{DLE}$-inequalities}\label{ssec:Inductive:Ineqs}

In the present subsection, we recall the definitions of {\em Sahlqvist} and {\em inductive} $\mathcal{L}_\mathrm{DLE}$-inequalities (cf.~\cite[Definition 3.4]{ConPal13}), which we will show to be preserved and reflected under suitable GMT translations in Section \ref{sec:Corresp:Via:Trans}. These definitions capture the non-classical counterparts, in each normal modal signature, of the classes of Sahlqvist (\cite{Sa75}) and inductive (\cite{GorankoV06}) formulas. The definition is given in terms of the order-theoretic properties of the interpretation of the logical connectives
(cf.\ \cite{CoGhPa14,ConPal12,ConPal13} for expanded discussions on the design principles of this definition). The fact that this notion applies uniformly across arbitrary normal modal signatures makes it possible to give a very general yet mathematically precise formulation to the question about the transfer of Sahlqvist-type results under GMT translations.

Technically speaking, these definitions are given parametrically in an order type. This contrasts to with the classical case, where the constantly-$1$ order-type is sufficient to encompass all Sahlqvist formulas. As a result, the preservation of the syntactic shape of each of these inequalities requires a GMT translation parametrized by the same order-type. These parametric translations will be introduced in Section \ref{Sec:ParamGMTTrans}.

\begin{definition}[\textbf{Signed Generation Tree}]
	\label{def: signed gen tree}
	The \emph{positive} (resp.\ \emph{negative}) {\em generation tree} of any $\mathcal{L}_\mathrm{DLE}$-term $s$ is defined by labelling the root node of the generation tree of $s$ with the sign $+$ (resp.\ $-$), and then propagating the labelling on each remaining node as follows:
	\begin{itemize}
		\item For any node labelled with $ \lor$ or $\land$, assign the same sign to its children nodes.
		\item For any node labelled with $h\in \mathcal{F}\cup \mathcal{G}$ of arity $n_h\geq 1$, and for any $1\leq i\leq n_h$, assign the same (resp.\ the opposite) sign to its $i$th child node if $\varepsilon_h(i) = 1$ (resp.\ if $\varepsilon_h(i) = \partial$).
	\end{itemize}
	Nodes in signed generation trees are \emph{positive} (resp.\ \emph{negative}) if are signed $+$ (resp.\ $-$).
\end{definition}

Signed generation trees will be mostly used in the context of term inequalities $s\leq t$. In this context we will typically consider the positive generation tree $+s$ for the left-hand side and the negative one $-t$ for the right-hand side.

For any term $s(p_1,\ldots p_n)$, any order type $\epsilon$ over $n$, and any $1 \leq i \leq n$, an \emph{$\epsilon$-critical node} in a signed generation tree of $s$ is a leaf node $+p_i$ with $\epsilon_i = 1$ or $-p_i$ with $\epsilon_i = \partial$. An $\epsilon$-{\em critical branch} in the tree is a branch from an $\epsilon$-critical node. The intuition is that variable occurrences corresponding to $\epsilon$-critical nodes are those that the algorithm ALBA will solve for in the process of eliminating them.

For every term $s(p_1,\ldots p_n)$ and every order type $\epsilon$, we say that $+s$ (resp.\ $-s$) {\em agrees with} $\epsilon$, and write $\epsilon(+s)$ (resp.\ $\epsilon(-s)$), if every leaf in the signed generation tree of $+s$ (resp.\ $-s$) is $\epsilon$-critical.
In other words, $\epsilon(+s)$ (resp.\ $\epsilon(-s)$) means that all variable occurrences corresponding to leaves of $+s$ (resp.\ $-s$) are to be solved for according to $\epsilon$. We will also write $+s'\prec \ast s$ (resp.\ $-s'\prec \ast s$) to indicate that the subterm $s'$ inherits the positive (resp.\ negative) sign from the signed generation tree $\ast s$. Finally, we will write $\epsilon(\gamma) \prec \ast s$ (resp.\ $\epsilon^\partial(\gamma) \prec \ast s$) to indicate that the signed subtree $\gamma$, with the sign inherited from $\ast s$, agrees with $\epsilon$ (resp.\ with $\epsilon^\partial$).
\begin{definition}
	\label{def:good:branch}
	Nodes in signed generation trees will be called \emph{$\Delta$-adjoints}, \emph{syntactically left residual (SLR)}, \emph{syntactically right residual (SRR)}, and \emph{syntactically right adjoint (SRA)}, according to the specification given in Table \ref{Join:and:Meet:Friendly:Table}.
	A branch in a signed generation tree $\ast s$, with $\ast \in \{+, - \}$, is called a \emph{good branch} if it is the concatenation of two paths $P_1$ and $P_2$, one of which may possibly be of length $0$, such that $P_1$ is a path from the leaf consisting (apart from variable nodes) only of PIA-nodes, and $P_2$ consists (apart from variable nodes) only of Skeleton-nodes. A good branch in which the nodes in $P_1$ are all SRA is called {\em excellent}.
	
	\begin{table}[h]
		\begin{center}
			\begin{tabular}{| c | c |}
				\hline
				Skeleton &PIA\\
				\hline
				$\Delta$-adjoints & SRA \\
				\begin{tabular}{ c c c c c c}
					$+$ &$\vee$ &$\wedge$ &$\phantom{\lhd}$ & &\\
					$-$ &$\wedge$ &$\vee$\\
					\hline
				\end{tabular}
				&
				\begin{tabular}{c c c c }
					$+$ &$\wedge$ &$g$ & with $n_g = 1$ \\
					$-$ &$\vee$ &$f$ & with $n_f = 1$ \\
					\hline
				\end{tabular}
				\\
				SLR &SRR\\
				\begin{tabular}{c c c c }
					$+$ & $\wedge$ &$f$ & with $n_f \geq 1$\\
					$-$ & $\vee$ &$g$ & with $n_g \geq 1$ \\
				\end{tabular}
				&\begin{tabular}{c c c c}
					$+$ &$\vee$ &$g$ & with $n_g \geq 2$\\
					$-$ & $\wedge$ &$f$ & with $n_f \geq 2$\\
				\end{tabular}
				\\
				\hline
			\end{tabular}
		\end{center}
		\caption{Skeleton and PIA nodes for $\mathrm{DLE}$.}\label{Join:and:Meet:Friendly:Table}
		\vspace{-1em}
	\end{table}
\end{definition}

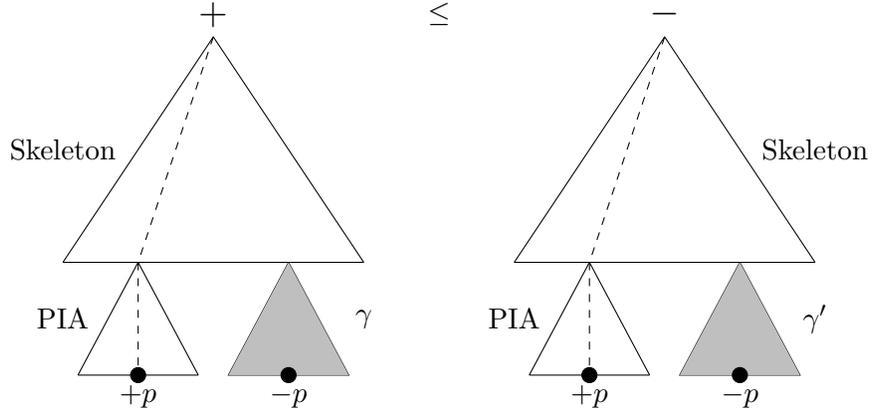
\begin{figure*}
	\begin{center}
		\begin{tikzpicture}
		\draw (-5,-1.5)   -- (-3,1.5) node[above]{\Large$+$} ;
		\draw (-5,-1.5) -- (-1,-1.5) ;
		\draw (-3,1.5) -- (-1,-1.5);
		\draw (-5,0) node{Skeleton}  ;
		\draw[dashed] (-3,1.5) -- (-4,-1.5);
		\draw (-4,-1.5) --(-4.8,-3);
		\draw (-4.8,-3) --(-3.2,-3);
		\draw (-3.2,-3) --(-4,-1.5);
		\draw[dashed] (-4,-1.5) -- (-4,-3);
		\draw[fill] (-4,-3) circle[radius=.1] node[below]{$+p$};
		
		\draw
		(-2,-1.5) -- (-2.8,-3) -- (-1.2,-3) -- (-2,-1.5);
		\fill[lightgray]
		(-2,-1.5) -- (-2.8,-3) -- (-1.2,-3);
		\draw (-1,-2.25)node{$\gamma$};
		\draw (-5,-2.25) node{PIA}  ;
		\draw (0,1.8) node{$\leq$};
		\draw (5,-1.5)   -- (3,1.5) node[above]{\Large$-$} ;
		\draw (5,-1.5) -- (1,-1.5) ;
		\draw (3,1.5) -- (1,-1.5);
		\draw (5,0) node{Skeleton}  ;
		\draw[dashed] (3,1.5) -- (2,-1.5);
		\draw (2,-1.5) --(2.8,-3);
		\draw (2.8,-3) --(1.2,-3);
		\draw (1.2,-3) --(2,-1.5);
		\draw[dashed] (2,-1.5) -- (2,-3);
		\draw[fill] (2,-3) circle[radius=.1] node[below]{$+p$};
		\draw
		(4,-1.5) -- (4.8,-3) -- (3.2,-3) -- (4, -1.5);
		\fill[lightgray]
		(4,-1.5) -- (4.8,-3) -- (3.2,-3) -- (4, -1.5);
		\draw (5,-2.25)node{$\gamma'$};
		\draw (1,-2.25) node{PIA} ;
		
		\draw[fill] (-2,-3) circle[radius=.1] node[below]{$-p$};
		\draw[fill] (4,-3) circle[radius=.1] node[below]{$-p$};
		\end{tikzpicture}
		\caption{A schematic representation of inductive inequalities.}
	\end{center}
\end{figure*}

\begin{definition}[Sahlqvist and Inductive inequalities]
	\label{Inducive:Ineq:Def}
	For any order type $\epsilon$, the signed generation tree $*s$ $(* \in \{-, + \})$ of a term $s(p_1,\ldots p_n)$ is \emph{$\epsilon$-Sahlqvist} if
	for all $1 \leq i \leq n$, every $\epsilon$-critical branch with leaf $p_i$ is excellent (cf.\ Definition \ref{def:good:branch}). An inequality $s \leq t$ is \emph{$\epsilon$-Sahlqvist} if the signed generation trees $+s$ and $-t$ are $\epsilon$-Sahlqvist. An inequality $s \leq t$ is \emph{Sahlqvist} if it is $\epsilon$-Sahlqvist for some  $\epsilon$.
	
	For any order type $\epsilon$ and any irreflexive and transitive relation $<_\Omega$ on $p_1,\ldots p_n$, the signed generation tree $*s$ $(* \in \{-, + \})$ of a term $s(p_1,\ldots p_n)$ is \emph{$(\Omega, \epsilon)$-inductive} if
	\begin{enumerate}
		\item for all $1 \leq i \leq n$, every $\epsilon$-critical branch with leaf $p_i$ is good (cf.\ Definition \ref{def:good:branch});
		\item every $m$-ary SRR-node occurring in the critical branch is of the form $ \circledast(\gamma_1,\dots,\gamma_{j-1},\beta,\gamma_{j+1}\ldots,\gamma_m)$, where for any $h\in\{1,\ldots,m\}\setminus j$: 
		\begin{enumerate}
			\item $\epsilon^\partial(\gamma_h) \prec \ast s$ (cf.\ discussion before Definition \ref{def:good:branch}), and
			%
			\item $p_k <_{\Omega} p_i$ for every $p_k$ occurring in $\gamma_h$ and for every $1\leq k\leq n$.
		\end{enumerate}

	\end{enumerate}
	
	We will refer to $<_{\Omega}$ as the \emph{dependency order} on the variables. An inequality $s \leq t$ is \emph{$(\Omega, \epsilon)$-inductive} if the signed generation trees $+s$ and $-t$ are $(\Omega, \epsilon)$-inductive. An inequality $s \leq t$ is \emph{inductive} if it is $(\Omega, \epsilon)$-inductive for some $<_\Omega$ and $\epsilon$.
\end{definition}
The definition above specializes so as to account for all the settings in which GMT translations have been defined. Below we expand on a selection of these.

\begin{example}[Intuitionistic language]\label{ex: frege ineq}
	
	As observed in \cite{ConPal12}, the Frege inequality
	
	\[p\rightarrow(q\rightarrow r)\leq (p\rightarrow q)\rightarrow (p\rightarrow r)\] 
	
	\noindent is not Sahlqvist for any order type, but is $(\Omega, \epsilon)$-inductive, e.g.\ for $r <_\Omega p <_\Omega q$ and $\epsilon(p, q, r) =(1, 1, \partial)$,  as can be seen from the signed generation trees below: 

	\begin{center}
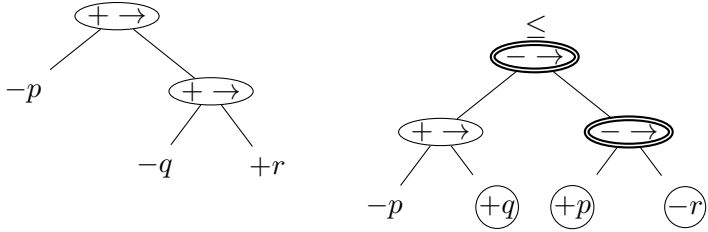
		\begin{multicols}{3}
			\begin{tikzpicture}
			\tikzstyle{level 1}=[level distance=1cm, sibling distance=2.5cm]
			\tikzstyle{level 2}=[level distance=1cm, sibling distance=1.5cm]
			\tikzstyle{level 3}=[level distance=1 cm, sibling distance=1.5cm]
			\node[PIA] {$+\rightarrow$}
			child{node{$-p$}}
			child{node[PIA]{$+\rightarrow$}
				child{node{$-q$}}
				child{node{$+r$}}}
			;
			
			\end{tikzpicture}
			
			\columnbreak
			$\leq$
			\columnbreak
			\begin{tikzpicture}
			\tikzstyle{level 1}=[level distance=1cm, sibling distance=2.5cm]
			\tikzstyle{level 2}=[level distance=1cm, sibling distance=1.5cm]
			\tikzstyle{level 3}=[level distance=1 cm, sibling distance=1.5cm]
			\node[Ske]{$-\rightarrow$}
			child{node[PIA]{$+\rightarrow$}
				child{node{$-p$}}
				child{node{\circled{$+q$}}}}
			child{node[Ske]{$-\rightarrow$}
				child{node{\circled{$+p$}}}
				child{node{\circled{$-r$}}}}
			;
			\end{tikzpicture}
		\end{multicols}
	\end{center}
	In the picture above, the circled variable occurrences are the $\epsilon$-critical ones, the nodes in double ellipses are Skeleton, and those in single ellipses are PIA. 		
\end{example}

\begin{example}[Bi-intuitionistic language]\label{ex: non inductive} In \cite[Section 4]{rauszer1974semi}, Rauszer  axiomatizes bi-intuitionistic logic considering the following  axioms among others, which we present in the form of inequalities:
	\[r\pdra (q\pdra p)\leq (p\vee q)\pdra p\quad\quad (q\pdra p)\rightarrow \bot\leq p\rightarrow q.\]
	The first inequality is not $(\Omega, \epsilon)$-inductive for any $\Omega$ and $\epsilon$; indeed, in the negative generation tree of $(p\vee q)\pdra p$, the variable  $p$ occurs in both subtrees rooted at the children of the root, which is  a binary SRR node, making it impossible to satisfy condition 2(b) of Definition \ref{Inducive:Ineq:Def} for any order-type $\epsilon$ and strict ordering $\Omega$.
	
	The second inequality is $\epsilon$-Sahlqvist for $\epsilon (p) = 1$ and $\epsilon (q) = \partial$, and is also $(\Omega, \epsilon)$-inductive but not Sahlqvist for $q<_{\Omega} p$ and $\epsilon(p) = \epsilon(q) = \partial$. It is  also $(\Omega, \epsilon)$-inductive but not Sahlqvist for $p<_{\Omega} q$ and $\epsilon(p) = \epsilon(q) = 1$.
\end{example} 	

\begin{example}[Intuitionistic  bi-modal language]
	The following Fischer Servi inequalities (cf.~\cite{fischerservi1977modal})
	\[\Diamond(q\rightarrow p)\leq \Box q\rightarrow \Diamond p \quad \quad \Diamond q\rightarrow \Box p\leq \Box (q\rightarrow p),\]
	are both $\varepsilon$-Sahlqvist for $\varepsilon(p) = \partial$ and $\varepsilon(q) = 1$, and are also both  $(\Omega, \varepsilon)$-inductive but not Sahlqvist for $p<_{\Omega} q$ and $\varepsilon(p) = \partial$ and $\varepsilon(q) = \partial$.
\end{example}

\begin{example}[Distributive modal  language]
	The following inequalities are key to Dunn's positive modal logic \cite{dunn1995positive}, the language of which is the $\{\lhd, \rhd\}$-free fragment of the language of Distributive Modal Logic \cite{GeNaVe05}:
	\[\Box q\wedge \Diamond p\leq \Diamond(q\wedge p) \quad \quad  \Box (q\vee p)\leq  \Diamond q\vee \Box p.\]
	The inequality on the left (resp.~right) is  $\varepsilon$-Sahlqvist for $\varepsilon(p) = \varepsilon(q) = 1$ (resp.~$\varepsilon(p) = \varepsilon(q) = \partial$), and is  $(\Omega, \varepsilon)$-inductive but not Sahlqvist for $p<_{\Omega} q$ and $\varepsilon(p) = 1$ and $\varepsilon(q) = \partial$ (resp.~$p<_{\Omega} q$ and $\varepsilon(p) = \partial$ and $\varepsilon(q) = 1$).
\end{example}

\section{The GMT translations}\label{semantic environment}

In this section we give a brief overview of some highlights in the history of the G\"{o}del-McKinsey-Tarski translation, its extensions and variations and the uses to which they have been put. We then recall the technical details of the translation and how its is founded upon the interplay of persistent and arbitrary valuations in intuitionistic Kripke frames.

\subsection{Brief history}
\label{ssec:brief history}
The GMT translation originates in the work of G\"{o}del \cite{Godel:Interp} and its algebraic underpinnings were analysed by McKinsey and Tarski in \cite{mckinsey1944algebra, mckinsey1946closed}. In particular, in \cite{mckinsey1944algebra} develops the theory of closure algebras, i.e.\ S4-algebras, while \cite{mckinsey1946closed} shows that every Brouwerian algebra embeds as the subalgebra of closed subsets of some closure algebra. In \cite{McKinseyTarski:Translation} this analysis is used to show that the GMT translation and a number of variations are full and faithful. This result is extended by Dummett and Lemmon \cite{dummett1959modal} in 1959 to all intermediate logics. In 1967 Grzegorczyk \cite{grzegorczyk1967some} showed that Dummett and Lemmon's result also holds if one replaces normal extensions of S4 with normal extension of what is now known as Grzegorczyk's logic. In light these developments, Maksimova and Rybakov \cite{maksimova1974lattice} launched a systematic study of the relationship between the lattice of intermediate logics and that of normal extensions of S4.

In 1976 Blok \cite{blok1976varieties} and Esakia \cite{esakia1976modal} independently built on this work to establish an isomorphism, based the GMT translation, between the lattice of intermediate logics and the lattice of normal extensions of Grzegorczyk's logic.

Between 1989 and 1992, several theorems are proven by Zakharyaschev and Chagrov about the preservation under GMT-based translations of properties such as decidability, finite model property, Kripke and Hallden completeness, disjunction property, and compactness. For detailed surveys on this line of research, the reader is referred to \cite{chagrov1992modal} and \cite{wolter2014blok}.

The first extension of the GMT translation to modal expansions of intuitionistic logic is introduced by Fischer Servi \cite{fischerservi1977modal} in 1977. Between 1979 and 1984 Shehtman and Sotirov also published on extensions (cf.\ discussions in \cite{chagrov1992modal, wolter2014blok}).

In the mid 90s, building on the work of Fischer Servi and Shehtman, Wolter and Zakharyaschev \cite{WoZa98} use an extension of the GMT translation to prove the transfer of a number of results, including finite model property, canonicity, decidability, tabularity and Kripke completeness, for intuitionistic modal logic with box only. In Remark \ref{rmk: comparison canonicity} we will expand on the relationship of their canonicity results (cf.\ \cite[Theorem 12]{WoZa98}) with those of the present paper. In \cite{WoZa97} this line of work is extended to intuitionistic modal logics with box and diamond. Duality theory is developed and the transfer of decidability, finite model property and tabularity under a suitably extended GMT translation is established. Moreover, a Blok-Esakia theorem is proved in this setting.

This research programme is further pursued by Wolter \cite{wolter1998CoImplication} in the setting bi-intuitionistic modal logic, obtained by adding the left residual of the disjunction (also known as co-implication) to the language of intuitionistic modal logic. He develops duality theory for these logics, extends the GMT translation and established a Blok-Esakia theorem.  Recent developments include work by G.\ Bezhanishvili who considers expansions of Prior's MIPC with  universal modalities, extends the GMT translation and establishes a Blok-Esakia theorem.

As the outline above shows, the research program on the transfer of results via GMT translations includes many transfers of completeness and canonicity (the latter in the form of d-persistence) which are closely related to the focus of the present paper. However, there are very few transfer results which specifically concern Sahlqvist theory, and this for the obvious reason that the formulation of such results depends on the availability of an independent definition of Sahlqvist formulas for each setting of (modal expansions of) non-classical logics, and these definitions have not been introduced in any of these settings, neither together nor in isolation, until 2005, when Gehrke, Nagahashi and Venema \cite{GeNaVe05} introduced the notion of Sahlqvist inequalities in the language of Distributive Modal Logic (DML).\footnote{
	We refer to \cite{CoGhPa14, ConPal13} for a systematic comparison between \cite[Definition 3.4]{GeNaVe05} and the definition of Sahlqvist and inductive inequalities for normal DLE-languages (cf.~Definition \ref{Inducive:Ineq:Def}).} This definition made it possible to formulate the  Sahlqvist correspondence theorem for DML-inequalities (cf.~\cite[Theorem 3.7]{GeNaVe05}) and prove it via reduction to a suitable classical poly-modal logic using GMT translations.

We conclude this brief survey by mentioning a recent paper by van Benthem, N.\ Bezhanishvili and Holliday \cite{vBBeHo16} which studies a GMT-like translation from modal logic with possibility semantics into bi-modal logic. They prove that this translation transfers and reflects first-order correspondence,  but point out that it destroys the Sahlqvist shape of formulas in all but a few special cases.

\subsection{Semantic analysis}\label{ssec: sem analysis godel tarski}
In the present section we recall the definition of the GMT translation in its original setting, highlight its basic semantic underpinning as a toggle between persistent and non-persistent valuations on S4-frames. This analysis will be extended  to a uniform account of the GMT translation for arbitrary normal DLE-logics in the next section.

In what follows, for any partial order $(W,\leq)$, we let $w{\uparrow}: = \{v\in W\mid w\leq v\}$, $w{\downarrow}: = \{v\in W\mid w\geq v\}$ for every $w\in W$, and for every $X\subseteq W$, we let $X{\uparrow}: = \bigcup_{x\in X} x{\uparrow}$ and $X{\downarrow}: = \bigcup_{x\in X} x{\downarrow}$. {\em Up-sets} (resp.\ {\em down-sets}) of $(W, \leq)$ are subsets $X\subseteq W$ such that $X = X{\uparrow}$ (resp.\ $X = X{\downarrow}$). We denote by $\mathcal{P}(W)$ the Boolean algebra of subsets of $W$, and by $\mathcal{P}^{\uparrow}(W)$ (resp.\ $\mathcal{P}^{\downarrow}(W)$) the (bi-)Heyting algebra of up-sets (resp.\ down-sets) of $(W, \leq)$. Finally we let $X^c$ denote the relative complement $W\setminus X$ of every subset $X\subseteq W$.

Fix a denumerable set $\mathsf{Prop}$ of propositional variables. The language of intuitionistic logic over $\mathsf{Prop}$ is given by
\[\mathcal{L}_I\ni\phi:: = p\mid \bot\mid \top\mid \phi\wedge \phi\mid \phi\vee\phi\mid \phi\rightarrow \phi.\]
The language of the normal modal logic S4 over $\mathsf{Prop}$ is given by
\[\mathcal{L}_{S4\Box}\ni\alpha:: = p\mid \bot \mid \alpha\vee \alpha\mid\alpha\wedge \alpha \mid\neg\alpha\mid \Box_{\leq}\alpha.\]

The GMT translation is the map $\tau\colon \mathcal{L}_I\rightarrow \mathcal{L}_{S4\Box}$ defined by the following recursion:
\begin{center}
	\begin{tabular}{r c l}
		$\tau (p)$ &  = & $\Box_{\leq} p$\\
		$\tau (\bot)$ &  = & $\bot$ \\
		$\tau (\top)$ &  = & $\top$ \\
		$\tau (\phi\wedge \psi)$ &  = & $\tau (\phi)\wedge \tau(\psi)$  \\
		$\tau (\phi\vee \psi)$ &  = & $\tau (\phi)\vee \tau(\psi)$  \\
		$\tau (\phi \rightarrow \psi)$ &  = & $\Box_{\leq}(\neg\tau (\phi)\vee \tau(\psi))$.  \\
	\end{tabular}
\end{center}


Both intuitionistic and S4-formulas can be interpreted on partial orders $\mathbb{F}  = (W, \leq)$, as follows: an S4-model is a tuple $(\mathbb{F}, U)$ where $U: \mathsf{Prop}\to \mathcal{P}(W)$ is a valuation. The interpretation $\Vdash^*$ of S4-formulas on S4-models is defined recursively as follows: for an $w\in W$,
\begin{center}
	\begin{tabular}{l l l}
		$\mathbb{F}, w, U \Vdash^* p$ &$\quad$& iff $p\in U(p)$\\
		$\mathbb{F}, w, U \Vdash^* \bot $ && never\\
		$\mathbb{F}, w, U \Vdash^* \top $ && always\\
		$\mathbb{F}, w, U \Vdash^* \alpha\wedge \beta$ &&iff $\ \mathbb{F}, w, U \Vdash^* \alpha\ $ and $\ \mathbb{F}, w, U \Vdash^* \beta$\\
		$\mathbb{F}, w, U \Vdash^* \alpha\vee \beta$ &&iff $\ \mathbb{F}, w, U \Vdash^* \alpha\ $ or $\ \mathbb{F}, w, U \Vdash^* \beta$\\
		$\mathbb{F}, w, U \Vdash^* \neg \alpha $ &&iff  $\ \mathbb{F}, w, U \not\Vdash^* \alpha\ $ \\
		$\mathbb{F}, w, U \Vdash^* \Box_\leq\alpha$ && iff   $\ \mathbb{F}, v, U \Vdash^* \alpha\ $ for any $v\in w{\uparrow}$.
	\end{tabular}
\end{center}
For any S4-formula $\alpha$ we let $\valb{\alpha}_U: = \{ w\mid \mathbb{F}, w, U\Vdash^*\alpha\}$. It is not difficult to verify that for every $\alpha\in \mathcal{L}_{S4}$ and any valuation $U$,
\begin{equation}\label{eq:extension of box}
\valb{\Box_\leq \alpha}_U = \valb{\alpha}_U^c{\downarrow}^c.
\end{equation}

An intuitionistic model is a tuple $(\mathbb{F}, V)$ where $V: \mathsf{Prop}\to \mathcal{P}^{\uparrow}(W)$ is a {\em persistent} valuation. The interpretation $\Vdash^*$ of S4-formulas on S4-models is defined recursively as follows: for an $w\in W$,
\begin{center}
	\begin{tabular}{l l l}
		$\mathbb{F}, w, V \Vdash p$ &$\quad$& iff $p\in V(p)$\\
		$\mathbb{F}, w, V \Vdash \bot \qquad$ &$\quad$& never\\
		$\mathbb{F}, w, V \Vdash \top \qquad$ &$\quad$& always\\
		$\mathbb{F}, w, V \Vdash \phi\wedge \psi$ &$\quad$& iff $\ \mathbb{F}, w, V \Vdash \phi\ $ and $\ \mathbb{F}, w, V \Vdash \psi$\\
		$\mathbb{F}, w, V \Vdash \phi\vee \psi$  &$\quad$& iff $\ \mathbb{F}, w, V \Vdash \phi\ $ or $\ \mathbb{F}, w, V \Vdash \psi$\\
		$\mathbb{F}, w, V \Vdash \phi \rightarrow \psi$ &$\quad$& iff  either  $\ \mathbb{F}, v, V \not\Vdash \phi\ $ or  $\ \mathbb{F}, v, V \Vdash \psi$ for any $v\in w{\uparrow}$.\\
	\end{tabular}
\end{center}
For any intuitionistic formula $\phi$ we let $\val{\phi}_V: = \{ w\mid \mathbb{F}, w, V\Vdash\phi\}$. It is not difficult to verify that for all $\phi, \psi\in \mathcal{L}_{I}$ and any persistent valuation $V$,
\begin{equation} \label{eq: extension of rightarrow}
\val{\phi\rightarrow \psi}_V = (\val{\phi}_V^c\cup \val{\psi}_V)^c{\downarrow}^c.
\end{equation}
Clearly, every persistent valuation $V$ on $\mathbb{F}$ is also a valuation on $\mathbb{F}$. Moreover, for every valuation $U$ on $\mathcal{F}$, the assignment mapping every $p\in \mathsf{Prop}$ to $U(p)^c{\downarrow}^c$ defines a persistent valuation $U^{\uparrow}$ on $\mathbb{F}$.
The main semantic property of the GMT translation is stated in the following well-known proposition:

\begin{prop}\label{fact: main property goedel transl}
	For every intuitionistic formula $\phi$ and every partial order $\mathbb{F}  = (W, \leq)$,
	\[\mathbb{F} \Vdash \phi \quad \mbox{ iff } \quad \mathbb{F} \Vdash^* \tau(\phi).\]
\end{prop}
\begin{proof}
	If $\mathbb{F} \not\Vdash \phi$, then $\mathbb{F}, w, V \not\Vdash \phi$ for some persistent valuation $V$ and $w\in W$. That is, $w\notin \val{\phi}_V = \valb{ \tau(\phi)}_V$, the last identity holding by item 1 of Lemma \ref{lemma: main property goedel transl}. Hence, $\mathbb{F}, w, V\not\Vdash^*\tau(\phi)$, i.e.\
	$\mathbb{F} \not\Vdash^* \tau(p)$. Conversely, if $\mathbb{F} \not\Vdash^* \tau(\phi)$, then $\mathbb{F}, w, U \not\Vdash \tau(\phi)$ for some valuation $U$ and $w\in W$. That is, $w\notin \valb{\tau(\phi)}_U =  \val{\phi}_{U^{\uparrow}}$, the last identity holding by item 2 of Lemma \ref{lemma: main property goedel transl}. Hence, $\mathbb{F}, w, U^{\uparrow}\not\Vdash \phi$, yielding $\mathbb{F} \not\Vdash \phi$.
\end{proof}

\begin{lemma}\label{lemma: main property goedel transl}
	For every intuitionistic formula $\phi$ and every partial order $\mathbb{F}  = (W, \leq)$,
	\begin{enumerate}
		\item $\val{\phi}_V = \valb{\tau(\phi)}_V$ for every persistent valuation $V$ on $\mathbb{F}$;
		\item $\valb{\tau(\phi)}_U = \val{\phi}_{U^{\uparrow}}$ for every valuation $U$ on $\mathbb{F}$.
	\end{enumerate}
\end{lemma}
\begin{proof}
	1. By induction on $\phi$. As for the base case, let $\phi: = p\in \mathsf{Prop}$. Then, for any persistent valuation $V$,
	\begin{center}
		\begin{tabular}{r c l l}
			$\val{p}_V $&$=$&$V(p)$ & (def.\ of $\val{\cdot}_V$)\\
			&$=$&$ V(p)^c{\downarrow}^c$ & ($V$ persistent)\\
			&$=$&$\valb{ \Box_{\leq} p}_V$& (equation \eqref{eq:extension of box})\\
			&$=$&$\valb{\tau(p)}_V$, & (def.\ of $\tau$)\\
		\end{tabular}
	\end{center}
	as required. As for the inductive step, let $\phi: = \psi\rightarrow \chi$. Then, for any persistent valuation $V$,
	\begin{center}
		\begin{tabular}{r c l l}
			$\val{\psi \rightarrow \chi}_V$&$=$&$(\val{\psi}_V^c \cup \val{\chi}_V)^c {\downarrow} ^c$& (equation \eqref{eq:extension of box})\\
			&$=$&$(\valb{\tau(\psi)}_V^c \cup \valb{\tau(\chi)}_V)^c {\downarrow}^c$& (induction hypothesis)\\
			&$=$&$\valb{\Box_\leq(\neg \tau(\psi)\vee \tau(\chi))}_V$& (equation \eqref{eq:extension of box}, def.\ of $\valb{\cdot}_V$)\\
			&$=$&$\valb{\tau(\psi\rightarrow \chi)}_V,$ & (def.\ of $\tau$)\\
		\end{tabular}
	\end{center}
	as required. The remaining cases are omitted.
	
	2. By induction on $\phi$. As for the base case, let $\phi: = p\in \mathsf{Prop}$.  Then, for any valuation $U$,
	\begin{center}
		\begin{tabular}{r c l l}
			$\valb{\tau(p)}_U$&$=$&$\valb{\Box_{\leq} p}_U$& (def.\ of $\tau$)\\
			&$=$&$\valb{p}_U ^c{\downarrow}^c$ & (equation \eqref{eq:extension of box})\\
			&$=$&$U(p)^c{\downarrow}^c$ & (def.\ of $\valb{\cdot}_U$)\\
			&$=$&$\val{p}_{U^{\uparrow}}$, & (def.\ of $U^{\uparrow}$) \\
		\end{tabular}
	\end{center}
	as required.
	As for the inductive step, let $\phi: = \psi\rightarrow \chi$. Then, for any valuation $U$,
	\begin{center}
		\begin{tabular}{r c l l}
			$\valb{\tau(\psi\rightarrow \chi)}_U$&$=$&$\valb{\Box_{\leq} (\neg \tau(\psi)\vee \tau(\chi))}_U$ & (def.\ of $\tau$)\\
			&$=$&$\valb{\neg \tau(\psi)\vee \tau(\chi)}_U ^c{\downarrow}^c$ & (equation \eqref{eq:extension of box})\\
			&$=$&$(\valb{\tau(\psi)}_U^c\cup \valb{\tau(\chi)}_U) ^c{\downarrow}^c$ & (def.\ of $\valb{\cdot}_U$) \\
			&$=$&$(\val{\psi}_{U^{\uparrow}}^c\cup \val{\chi}_{U^{\uparrow}}) ^c{\downarrow}^c$ & (induction hypothesis)\\
			&$=$&$\val{\psi\rightarrow \chi}_{U^{\uparrow}}$, &  (equation \eqref{eq: extension of rightarrow}, $U^{\uparrow}$ persistent)\\
		\end{tabular}
	\end{center}
	as required.
	The remaining cases are omitted.
\end{proof}

Hence, the main semantic property of GMT translation, stated in Proposition \ref{fact: main property goedel transl}, can be understood in terms of the interplay between persistent and nonpersistent valuations, as captured in the above lemma. In the next section, we are going to establish a general template for this interplay and then apply it to the algebraic analysis of GMT translations in various  settings.  

\section{Unifying analysis of GMT translations}\label{ssec:template} 

In the present section, we  generalize the key mechanism captured in Section \ref{ssec: sem analysis godel tarski} and guaranteeing the preservation and reflection of validity under the GMT translation. Being able to identify this pattern in generality will make it possible to recognise this mechanism in several logical settings, as we do in the following section.

Let $\mathcal{L}_1$ and $\mathcal{L}_2$ be  propositional languages over a given set $X$, and let $\mathbb{A}$ and $\bbB$ be ordered $\mathcal{L}_1$- and $\mathcal{L}_2$-algebras respectively, such that an order-embedding
$e \colon \bbA \hookrightarrow \bbB$ exists. For each $V \in \bbA^X$ and $U \in \bbB^X$, let  $\val{\cdot}_V$ and $\valb{\cdot}_U$ denote their unique homomorphic extensions to $\mathcal{L}_1$ and $\mathcal{L}_2$ respectively. Clearly, $e \colon \bbA \hookrightarrow \bbB$ lifts to a map $\overline{e}: \bbA^X \to \bbB^X$ by the assignment $V\mapsto e\circ V$.  
\begin{prop}\label{prop: main prop of Godel tarski algebraically}
	Let $\tau \colon \mathcal{L}_1 \to \mathcal{L}_2$ and  $r\colon \bbB^X\to \bbA^X$ be such that the following conditions hold for every $\phi \in \mathcal{L}_1$:
	\begin{enumerate}[label={(\alph*)}]
		\item $e(\val{ \phi}_V) = \valb{\tau(\phi)}_{\overline{e}(V)}$ for every $V \in \bbA^X$;
		\item $\valb{ \tau (\phi) }_U = e(\val{\phi}_{r(U)})$ for every $U \in \bbB^X$.
	\end{enumerate}
	Then, for all $\phi, \psi\in \mathcal{L}_1$,
	\[
	\bbA \models \phi\leq \psi \quad \mbox{ iff } \quad \bbB \models \tau(\phi)\leq \tau(\psi).
	\]
\end{prop}
\begin{proof}
	From left to right, suppose contrapositively that $(\bbB,U) \not\models \tau(\phi)\leq \tau(\psi)$ for some $U \in \bbB^X$, that is, $\valb{\tau(\phi)}_U\not\leq \valb{\tau(\psi)}_U$.  By item (b) above,  this non-inequality is equivalent to $e(\val{\phi}_{r(U)})\not\leq e(\val{\psi}_{r(U)})$, which, by the monotonicity of $e$, implies that $\val{\phi}_{r(U)}\not\leq \val{\psi}_{r(U)}$, that is,  $(\bbA, r(U)) \not\models \phi\leq \psi$, as required. Conversely, if $(\bbA,V) \not\models \phi\leq \psi$ for some $V \in \bbA^X$, then $\val{\phi}_{V}\not\leq \val{\psi}_{V}$, and hence, since  $e$ is an order-embedding and by item (a) above, $\valb{\tau(\phi)}_{\overline{e}(V)} = e(\val{\phi}_{V})\not\leq e(\val{\psi}_{V}) = \valb{\tau(\psi)}_{\overline{e}(V)}$, that is   $(\bbB, \overline{e}(V)) \not\models \tau(\phi)\leq \tau(\psi)$, as required.
\end{proof}
In the proof above we have only made use of the assumption that $e$ is an order-embedding, but have not needed to assume any property of $r$.
Moreover, the proposition above is independent of the logical/algebraic signature of choice, and hence can be used as a general template accounting for the main property of GMT translations. Finally, the proposition holds for {\em arbitrary} ordered algebras. This latter point is key to the treatment of Sahlqvist canonicity via translation.

\section{Instantiations of the general template for GMT translations}\label{sec:instantiations}

In the present section, we look into a family of GMT-type translations, defined for different logics, to which we  apply the template of Section \ref{ssec:template}. We organize these considerations into two subsections, in the first of which we treat the GMT translations for intuitionistic and co-intuitionistic logic and discuss how these extend to bi-intuitionistic logic. In the second subsection we consider parametrized versions of the GMT translation in the style of Gehrke, Nagahashi and Venema \cite{GeNaVe05} for normal modal expansions of bi-intuitionistic and DLE-logics.

\subsection{Non-parametric GMT translations}

As is well known, the semantic underpinnings of the GMT translations for intuitionistic and co-intuitionistic logic are the embeddings of Heyting algebras as the algebras of open elements of interior algebras and of Brouwerian algebras (aka co-Heyting algebras) as algebras of closed elements of closure algebras \cite{mckinsey1946closed}. In the following subsections, the existence of these embeddings will be used to satisfy the key assumptions of Section \ref{ssec:template}.

\subsubsection{Interior operator analysis of the GMT translation for intuitionistic logic}

As observed above, Proposition \ref{prop: main prop of Godel tarski algebraically} generalizes Proposition \ref{fact: main property goedel transl} in more than one way.  In the present subsection, we show that the GMT translation for intuitionistic logic verifies the conditions of Proposition \ref{prop: main prop of Godel tarski algebraically}. This is an alternative proof of the well-known fact that the GMT-translation is full and faithful, not only with respect to perfect algebras (dual to frames), but also with respect to general algebras. This fact is necessary for the proof of the transfer of Sahlqvist canonicity (cf.~Section \ref{sec:canonicity}).
Towards this goal, we let $X: = \mathsf{Prop}$, $\mathcal{L}_1: = \mathcal{L}_I$, and $\mathcal{L}_2: = \mathcal{L}_{S4}$. Moreover, we let $\bbA$ be a  Heyting algebra,  and $\bbB$  a Boolean algebra  such that an order-embedding $e\colon \bbA\hookrightarrow\bbB$ exists, which is  also a homomorphism of the lattice reducts of $\bbA$ and $\bbB$, and has  a right adjoint\footnote{That is, $e(a)\leq b$ iff $a\leq \iota(b)$ for every $a\in \bbA$ and $b\in \bbB$. By well known order-theoretic facts (cf.\ \cite{DaPr90}), $e\circ\iota$ is an {\em interior operator}, that is, for every $b, b'\in \bbB$, \begin{enumerate}
		\item[i1.] $(e\circ\iota)(b)\leq b$;
		\item[i2.] if $b\leq b'$ then $(e\circ\iota)(b)\leq (e\circ\iota)(b')$;
		\item[i3.] $(e\circ\iota)(b)\leq (e\circ\iota)((e\circ\iota)(b))$.
	\end{enumerate}Moreover, $e\circ\iota\circ e = e$ and $\iota = \iota\circ e \circ \iota$ (cf.\ \cite[Lemma 7.26]{DaPr90}).} $\iota\colon \bbB\to \bbA$ such that for all $a, b\in \bbA$,
\begin{equation}
\label{eq: rightarrow Rightarrow e and iota}
a\rightarrow^\bbA b = \iota(\neg^\bbB e(a)\vee^\bbB e(b)).\end{equation}
Then $\bbB$ can be endowed with a natural structure of Boolean algebra expansion (BAE) by defining $\Box^\bbB\colon \bbB\to \bbB$ by the assignment $b\mapsto (e\circ \iota)(b)$.
The following is a well known fact in  algebraic modal logic:
\begin{prop}\label{prop:BAE is s4}
	The BAE $(\bbB, \Box^\bbB)$, with $\Box^\bbB$ defined above, is  normal  and is also an S4-modal algebra.
\end{prop}
\begin{proof}
	The fact that $\Box^\bbB$ preserves finite (hence empty) meets readily follows from the fact that $\iota$ is a right adjoint, and hence preserves existing (thus all finite) meets of $\bbB$, and $e$ is a lattice homomorphism. For every $b\in \bbB$, $\iota(b)\leq \iota(b)$ implies that $\Box^\bbB b = e(\iota(b))\leq b$, which proves (T). For every $b\in \bbB$, $e(\iota(b))\leq e(\iota(b))$ implies that $\iota(b)\leq \iota(e(\iota(b)))$ and hence $\Box^\bbB b = e(\iota(b))\leq e(\iota(e(\iota(b)))) = (e\circ\iota)((e\circ\iota)(b)) = \Box^\bbB\Box^\bbB b$, which proves K4.
\end{proof}
Finally, we let $r: \bbB^X\to \bbA^X$ be defined by the assignment $U\mapsto (\iota\circ U)$.

\begin{prop}
	\label{prop:tau verifies a and b}
	Let $\bbA$, $\bbB$, $e\colon \bbA\hookrightarrow\bbB$ and $r\colon \bbB^X\to \bbA^X$ be as above.\footnote{The assumption that  $e$ is a homomorphism of the lattice reducts of $\bbA$ and $\bbB$ is needed for the inductive steps relative to $\bot, \top, \wedge, \vee$  in the proof this proposition, while condition \eqref{eq: rightarrow Rightarrow e and iota} is needed for the step relative to $\rightarrow$.} Then the GMT translation $\tau$ satisfies conditions (a) and (b) of Proposition \ref{prop: main prop of Godel tarski algebraically} for any formula $\phi\in \mathcal{L}_I$.
\end{prop}
\begin{proof}
	By induction on $\phi$. As for the base case, let $\phi: = p\in \mathsf{Prop}$. Then, for any $U\in \bbB^X$ and   $V\in\bbA^X$,
	\begin{center}
		\begin{tabular}{r c l c r c l l}
			$e(\val{p}_{r(U)})$ & = & $e((\iota\circ U)(p))$ &\quad\quad &$\valb{\tau(p)}_{\overline{e}(V)}$ & = & $\valb{\Box_{\leq} p}_{\overline{e}(V)}$ \\
			& = & $(e\circ \iota)(\valb{p}_U)$ &assoc.\ of $\circ$ && = & $\Box^\bbB\valb{p}_{\overline{e}(V)}$ \\
			& = & $\Box^\bbB\valb{p}_U$ &\quad\quad && = & $\Box^\bbB((e\circ V)(p))$ \\
			& = & $\valb{\Box_{\leq} p}_U$ &\quad\quad && = & $(e\circ \iota)((e\circ V)(p))$ \\
			& = & $\valb{\tau( p)}_U$. &\quad\quad && = & $e((\iota \circ e)(V(p)))$ & assoc.\ of $\circ$\\
			&&&&& = & $e(V(p))$ & $e\circ(\iota\circ e) =e$\\
			&&&&& = & $e(\val{p}_V)$, \\
		\end{tabular}
	\end{center}
	which proves  the base cases of  (b) and (a) respectively. As for the inductive step, let $\phi: = \psi\rightarrow \chi$. Then, for any $U\in \bbB^X$ and   $V\in\bbA^X$,
	\begin{center}
		\begin{tabular}{r c l  l}
			$e(\val{\psi\rightarrow \chi}_{r(U)})$ & = & $e(\val{\psi}_{r(U)}\rightarrow^\bbA \val{ \chi}_{r(U)})$ \\
			& = & $e(\iota(\neg^\bbB e(\val{\psi}_{r(U)})\vee^\bbB e(\val{ \chi}_{r(U)})))$ & assumption \eqref{eq: rightarrow Rightarrow e and iota} \\
			& = & $e(\iota(\neg^\bbB\valb{\tau(\psi)}_{U}\vee^\bbB \valb{\tau(\chi)}_{U}))$ & (induction hypothesis)\\
			& = & $(e\circ \iota)(\neg^\bbB\valb{\tau(\psi)}_{U}\vee^\bbB \valb{\tau(\chi)}_{U})$ & \\
			& = & $\Box^\bbB(\neg^\bbB \valb{\tau(\psi)}_{U}\vee^\bbB \valb{ \tau(\chi)}_{U})$ & \\
			& = & $\valb{\Box_\leq(\neg \tau(\psi)\vee\tau(\chi))}_{U}$ & \\
			& = & $\valb{\tau(\psi\rightarrow\chi)}_{U}$. & \\
		\end{tabular}
	\end{center}
	
	\begin{center}
		\begin{tabular}{r c l  l}
			$e(\val{\psi\rightarrow \chi}_{V})$ & = & $e(\val{\psi}_{V}\rightarrow^\bbA \val{ \chi}_{V})$ \\
			& = & $e(\iota(\neg^\bbB e(\val{\psi}_{V})\vee^\bbB e(\val{ \chi}_{V})))$ & assumption \eqref{eq: rightarrow Rightarrow e and iota} \\
			& = & $e(\iota(\neg^\bbB \valb{\tau(\psi)}_{\overline{e}(V)}\vee^\bbB \valb{ \tau(\chi)}_{\overline{e}(V)}))$ & (induction hypothesis)\\
			& = & $(e\circ \iota)(\neg^\bbB \valb{\tau(\psi)}_{\overline{e}(V)}\vee^\bbB \valb{ \tau(\chi)}_{\overline{e}(V)})$ & \\
			& = & $\Box^\bbB(\neg^\bbB \valb{\tau(\psi)}_{\overline{e}(V)}\vee^\bbB \valb{ \tau(\chi)}_{\overline{e}(V)})$ & \\
			& = & $\valb{\Box_\leq(\tau(\psi)\vee\tau(\chi))}_{\overline{e}(V)}$ & \\
			& = & $\valb{\tau(\psi\rightarrow\chi)}_{\overline{e}(V)}$. & \\
		\end{tabular}
	\end{center}
	The remaining cases are straightforward, and are left to the reader.
\end{proof}


The following strengthening of Proposition \ref{fact: main property goedel transl} immediately follows from Propositions \ref{prop: main prop of Godel tarski algebraically} and \ref{prop:tau verifies a and b}:
\begin{cor}
	Let $\bbA$ be a Heyting algebra and $\bbB$ a Boolean algebra such that $e\colon \bbA\hookrightarrow\bbB$ and $\iota\colon \bbB\to \bbA$ exist as above. Then for all intuitionistic formulas $\phi$ and $\psi$,
	\[
	\bbA \models \phi\leq \psi \quad \mbox{ iff } \quad \bbB \models \tau(\phi)\leq \tau(\psi),
	\]
	where $\tau$ is the GMT translation.
\end{cor}

We conclude this subsection by showing that 
the required embedding of Heyting algebras into Boolean algebras exists. This well known fact appears in its order-dual version already in \cite[Theorem 1.15]{mckinsey1946closed}, where is shown in a purely algebraic way. We give an alternative proof, based on very well known duality-theoretic facts.     
The format of the statement is different from the dual of \cite[Theorem 1.15]{mckinsey1946closed} both because it is in the form required by Proposition \ref{prop: main prop of Godel tarski algebraically}, and also because it highlights that the embedding lifts to the canonical extensions of the algebras involved, where it acquires additional properties. This  is relevant to the analysis of canonicity via translation in Section \ref{sec:canonicity}.

\begin{prop}\label{prop:existence of algebras and adjoints:Heyting}
	For every Heyting algebra $\bbA$, there exists a Boolean algebra $\bbB$ such that $\bbA$ embeds into $\bbB$ via some order-embedding $e\colon \bbA\hookrightarrow\bbB$ which is also a homomorphism of the lattice reducts of $\bbA$ and $\bbB$ and has a right adjoint $\iota\colon \bbB\to \bbA$ verifying condition \eqref{eq: rightarrow Rightarrow e and iota}. Finally, these facts lift to the canonical extensions of $\bbA$ and $\bbB$ as in the following diagram:
	\begin{center}
		
		\begin{tikzpicture}[node/.style={circle, draw, fill=black}, scale=1]
		
		\node (A) at (-1.5,-1.5) {$\mathbb{A}$};
		\node (A delta) at (-1.5,1.5) {$\mathbb{A}^{\delta}$};
		\node (B) at (1.5,-1.5) {$\mathbb{B}$};
		\node (B delta) at (1.5,1.5) {$\mathbb{B}^{\delta}$};
		
		\draw [right hook->] (A) to  (A delta);
		\draw [right hook->] (B)  to (B delta);
		\draw [right hook->] (A)  to node[below] {$e$}  (B);
		\draw [right hook->] (A delta)  to node[below] {$e^{\delta}$}  (B delta);
		
		\draw [->] (B delta) to [out=135,in=45, looseness=0.5]  node[below] {\rotatebox[origin=c]{270}{$\vdash$}} node[above] {$\iota^{\pi}$}  (A delta);
		\draw [->] (B) to [out=135,in=45, looseness=0.5]  node[below] {\rotatebox[origin=c]{270}{$\vdash$}} node[above] {$\iota$}  (A);
		
		\draw [->] (B delta) to [out=225,in=-45, looseness=1]  node[above] {\rotatebox[origin=c]{270}{$\vdash$}} node[below] {$c$}  (A delta);
		\end{tikzpicture}
	\end{center}
\end{prop}

\begin{proof}
	Via  Esakia duality \cite{Es75}, the Heyting algebra $\bbA$ can be identified  with the algebra of clopen up-sets of its associated Esakia space $\mathbb{X}_\bbA$, which is a Priestley space, hence a Stone space. Let $\bbB$ be the Boolean algebra of the clopen sub{\em sets} of $\mathbb{X}_\bbA$. Since any clopen up-set is in particular a clopen subset,  a natural order embedding $e: \bbA\hookrightarrow\bbB$ exists, which is also a lattice homomorphism between $\bbA$ and $\bbB$. This shows the first part of the claim.
	
	As to the second part, notice that Esakia spaces are Priestley spaces in which the downward-closure of a clopen set is a clopen set.
	
	Therefore, we can define the map  $\iota\colon \bbB\to \bbA$ by the assignment $b\mapsto \neg((\neg b){\downarrow})$ where $b$ is identified with its corresponding clopen set in $\mathbb{X}_\bbA$, $\neg b$ is identified with the relative complement of the clopen set $b$, and $(\neg b){\downarrow}$ is defined using the order in $\mathbb{X}_\bbA$. It can be readily verified that $\iota$ is the right adjoint of $e$ and that moreover  condition \eqref{eq: rightarrow Rightarrow e and iota} holds. 
	%
	
	Finally, $e:\bbA\rightarrow\bbB$ being also a homomorphism between the lattice reducts of $\bbA$ and $\bbB$ implies that $e$ is smooth and its canonical extension $e^\delta\colon \bbad\to \bbbd$, besides being an order-embedding, is a complete homomorphism between the lattice reducts of $\bbad$ and $\bbbd$ (cf.\  \cite[Corollary 4.8]{GeHa01}), and hence is endowed with both a left and a right adjoint. Furthermore, the right adjoint of $e^\delta$ coincides with $\iota^\pi$ (cf.\  \cite[Proposition 4.2]{GePr07}). Hence, $\bbbd$ can be endowed with a natural structure of S4 bi-modal algebra by defining $\Box_{\leq}^{\bbbd}\colon \bbbd\to \bbbd$ by the assignment $u\mapsto (e^\delta\circ \iota^\pi)(u)$, and $\Diamond_{\geq}^{\bbbd}\colon \bbbd\to \bbbd$ by the assignment $u\mapsto (e^\delta\circ c)(u)$.
\end{proof}

\subsubsection{The GMT translation for co-intuitionistic logic}
\label{ssec: sem analysis co-godel tarski}

In the present subsection, we  show that the GMT translation for co-intuitionistic logic (which we will sometimes refer to as the co-GMT translation) verifies the conditions of Proposition \ref{prop: main prop of Godel tarski algebraically}. Our presentation is a straightforward dualization of the previous subsection. We include it for the sake of completeness and for introducing some notation.   Semantically, this dualization involves replacing Heyting algebras with co-Heyting algebras (aka Brouwerian algebras cf.~\cite{mckinsey1946closed}), and interior algebras (aka S4 algebras with box) with closure algebras (aka S4 algebras with diamond).

Fix a denumerable set $\mathsf{Prop}$ of propositional variables. The language of co-intuitionistic logic over $\mathsf{Prop}$ is given by
\[\mathcal{L}_C\ni\phi:: = p\mid \bot\mid \top\mid \phi\wedge \phi\mid \phi\vee\phi\mid \phi\pdra \phi.\]
The target language for translating co-intuitionistic logic is that of the normal modal logic S4$\Diamond$ over $\mathsf{Prop}$, given by
\[\mathcal{L}_{S4\Diamond}\ni\alpha:: = p\mid \bot \mid \alpha\vee \alpha\mid\alpha\wedge \alpha \mid\neg\alpha\mid \Diamond_{\geq}\alpha.\]
Just like intuitionistic logic, formulas of co-intuitionistic logic can be interpreted on partial orders $\mathbb{F} = (W,\leq)$ with persistent valuations. Here we only report on the interpretation of $\Diamond_{\geq}$-formulas in $\mathcal{L}_{S4\Diamond}$ and  $\pdra$-formulas in $\mathcal{L}_C$:
\begin{center}
	\begin{tabular}{l c l}
		$\mathbb{F}, w, U \Vdash^* \Diamond_{\geq}\phi$ &$\quad$& iff   $\ \mathbb{F}, v, U \Vdash^* \phi\ $ for some $v\in w{\downarrow}$.\\
		$\mathbb{F}, w, V \Vdash \phi \pdra \psi$ &$\quad$& iff    $\ \mathbb{F}, v, V \not\Vdash \phi\ $ and  $\ \mathbb{F}, v, V \Vdash \psi$ for some $v\in w{\downarrow}$.\\
	\end{tabular}
\end{center}
The language $\mathcal{L}_C$ is naturally interpreted in co-Heyting algebras. The connective $\pdra$ is interpreted as the left residual of $\vee$.
The co-GMT translation is the map $\sigma\colon \mathcal{L}_C\rightarrow \mathcal{L}_{S4\Diamond}$ defined by the following recursion:
\begin{center}
	\begin{tabular}{r c l}
		$\sigma (p)$ &  = & $\Diamond_{\geq} p$\\
		$\sigma (\bot)$ &  = & $\bot$ \\
		$\sigma (\top)$ &  = & $\top$ \\
		$\sigma (\phi\wedge \psi)$ &  = & $\sigma (\phi)\wedge \sigma(\psi)$  \\
		$\sigma (\phi\vee \psi)$ &  = & $\sigma(\phi)\vee \sigma(\psi)$  \\
		$\sigma (\phi \pdra \psi)$ &  = & $\Diamond_{\geq}(\neg\sigma (\phi) \wedge \sigma(\psi))$  \\
	\end{tabular}
\end{center}

Next, we show that Proposition \ref{prop: main prop of Godel tarski algebraically} applies to the co-GMT translation.
We let $X: = \mathsf{Prop}$, $\mathcal{L}_1: = \mathcal{L}_C$, and $\mathcal{L}_2: = \mathcal{L}_{S4\Diamond}$. Moreover, we let $\bbA$ be a  co-Heyting algebra,  and $\bbB$  a Boolean algebra  such that an order-embedding $e\colon \bbA\hookrightarrow\bbB$ exists, which is  also a homomorphism of the lattice reducts of $\bbA$ and $\bbB$, and has  a left adjoint\footnote{That is, $c(b)\leq a$ iff $b\leq e(a)$ for every $a\in \bbA$ and $b\in \bbB$. By well known order-theoretic facts (cf.\ \cite{DaPr90}), $e\circ c$ is an {\em interior operator}, that is, for every $b, b'\in \bbB$, \begin{enumerate}[label={c\arabic*.}]
		\item $b\leq (e\circ c)(b)$;
		\item if $b\leq b'$ then $(e\circ c)(b)\leq (e\circ c)(b')$;
		\item $(e\circ c)((e\circ c)(b))\leq (e\circ c)(b)$.
	\end{enumerate}Moreover, $e\circ c\circ e = e$ and $c = c\circ e \circ c$ (cf.\ \cite[Lemma 7.26]{DaPr90}).} $c\colon \bbB\to \bbA$ such that for all $a, b\in \bbA$,
\begin{equation}
\label{eq: pdra minus e and c}
a\pdra^\bbA b = c(\neg^\bbB e(a) \wedge^\bbB e(b)).\end{equation}
Then $\bbB$ can be endowed with a natural structure of Boolean algebra expansion (BAE) by defining $\Diamond^\bbB\colon \bbB\to \bbB$ by the assignment $b\mapsto (e\circ c)(b)$.
The following is the dual of Proposition \ref{prop:BAE is s4} and its proof is omitted.
\begin{prop}\label{prop:BAE dia is s4}
	The BAE $(\bbB, \Diamond^\bbB)$, with $\Diamond^\bbB$ defined above, is  normal  and is also an $S4\Diamond$-modal algebra.
\end{prop}
Finally, we let $r: \bbB^X\to \bbA^X$ be defined by the assignment $U\mapsto (c\circ U)$. The proof of the following proposition is similar to that of Proposition \ref{prop:tau verifies a and b}, and its proof is omitted.

\begin{prop}
	\label{prop:sigma verifies a and b}
	Let $\bbA$, $\bbB$, $e\colon \bbA\hookrightarrow\bbB$ and $r\colon \bbB^X\to \bbA^X$ be as above.\footnote{The assumption that  $e$ is a homomorphism of the lattice reducts of $\bbA$ and $\bbB$ is needed for the inductive steps relative to $\bot, \top, \wedge, \vee$  in the proof this proposition, while condition \eqref{eq: pdra minus e and c} is needed for the step relative to $\pdra$.} Then the co-GMT translation $\sigma$ satisfies conditions (a) and (b) of Proposition \ref{prop: main prop of Godel tarski algebraically} for any formula $\phi\in \mathcal{L}_C$.
\end{prop}

The following corollary immediately follows from Propositions \ref{prop: main prop of Godel tarski algebraically} and \ref{prop:sigma verifies a and b}:
\begin{cor}\label{cor:main property sigma}
	Let $\bbA$ be a co-Heyting algebra and $\bbB$ a Boolean algebra such that an order-embedding $e\colon \bbA\hookrightarrow\bbB$ exists, which is  a homomorphism of the lattice reducts of $\bbA$ and $\bbB$, and has  a left adjoint $c\colon \bbB\to \bbA$ such that condition \eqref{eq: pdra minus e and c} holds for all $a, b\in \bbA$. Then for all $\phi,\psi\in \mathcal{L}_C$,
	\[
	\bbA \models \phi\leq \psi \quad \mbox{ iff } \quad \bbB \models \sigma(\phi)\leq \sigma(\psi),
	\]
	where $\sigma$ is the co-GMT translation.
\end{cor}

As in the previous subsection, we conclude by giving a version of \cite[Theorem 1.15]{mckinsey1946closed} which we need to instantiate Proposition \ref{prop: main prop of Godel tarski algebraically} and for the analysis of canonicity via translation in Section \ref{sec:canonicity}. The proof is dual to that of Proposition \ref{prop:existence of algebras and adjoints:Heyting}.

\begin{prop}\label{prop:existence of algebras and adjoints:coHeyting}
	For every co-Heyting algebra $\bbA$, there exists a Boolean algebra $\bbB$ such that $\bbA$ embeds into $\bbB$ via some order-embedding $e\colon \bbA\hookrightarrow\bbB$ which is  a homomorphism of the lattice reducts of $\bbA$ and $\bbB$, and has a left adjoint $c\colon \bbB\to \bbA$ verifying condition \eqref{eq: pdra minus e and c}. Finally, these facts lift to the canonical extensions of $\bbA$ and $\bbB$ as in the following diagram:
	\begin{center}
		\begin{tikzpicture}[node/.style={circle, draw, fill=black}, scale=1]
		
		\node (A) at (-1.5,-1.5) {$\mathbb{A}$};
		\node (A delta) at (-1.5,1.5) {$\mathbb{A}^{\delta}$};
		\node (B) at (1.5,-1.5) {$\mathbb{B}$};
		\node (B delta) at (1.5,1.5) {$\mathbb{B}^{\delta}$};
		
		\draw [right hook->] (A) to  (A delta);
		\draw [right hook->] (B)  to (B delta);
		\draw [right hook->] (A)  to node[above] {$e$}  (B);
		\draw [right hook->] (A delta)  to node[above] {$e^{\delta}$}  (B delta);
		
		\draw [->] (B delta) to [out=225,in=-45, looseness=0.5]  node[above] {\rotatebox[origin=c]{270}{$\vdash$}} node[below] {$c^{\sigma}$}  (A delta);
		\draw [->] (B) to [out=225,in=-45, looseness=0.5]  node[above] {\rotatebox[origin=c]{270}{$\vdash$}} node[below] {$c$}  (A);
		\draw [->] (B delta) to [out=135,in=45, looseness=1]  node[below] {\rotatebox[origin=c]{270}{$\vdash$}} node[above] {$\iota$}  (A delta);
		\end{tikzpicture}
	\end{center}
\end{prop}

\subsubsection{Extending the GMT and co-GMT translations to bi-intuitionistic logic}
\label{sssec: non parametric bi}
In the present subsection we consider the extensions of the GMT and co-GMT translations to  bi-intuitionistic logic (aka Heyting-Brouwer logic according to the terminology of Rauszer \cite{rauszer1974semi} who introduced this logic in the same paper, see also \cite{rauszer1974formalization, rauszer1980algebraic}). The extension $\tau'$ of the GMT translation considered below coincides with the one introduced by Wolter in \cite{wolter1998CoImplication} restricted to the language of bi-intuitionistic logic. The paper \cite{wolter1998CoImplication} considers a modal expansion of bi-intuitionistic logic with box and diamond operators where the extended GMT translation is used to establish a Blok-Esakia result and to transfer properties such as completeness, finite model property and decidability.

The language of bi-intuitionistic logic is given by
\[\mathcal{L}_B\ni\phi:: = p\mid \bot\mid \top\mid \phi\wedge \phi\mid \phi\vee\phi\mid\phi\rightarrow \phi\mid \phi\pdra \phi\]
The language of the normal bi-modal logic S4 is  given by
\[\mathcal{L}_{S4B}\ni\alpha:: = p\mid \bot \mid \alpha\vee \alpha\mid \neg\alpha\mid \Box_{\leq} \alpha\mid \Diamond_{\geq}\alpha\]
The GMT and the co-GMT translations $\tau$ and $\sigma$ can be extended to the bi-intuitionistic language as the maps $\tau', \sigma'\colon \mathcal{L}_B\rightarrow \mathcal{L}_{S4B}$ defined by the following recursions:
\begin{center}
	\begin{tabular}{r c l c r c l}
		$\tau' (p)$ &  = & $\Box_{\leq} p$ & $\qquad$ & $\sigma' (p)$ &  = & $\Diamond_{\geq} p$\\
		$\tau' (\bot)$ &  = & $\bot$ & & $\sigma' (\bot)$ &  = & $\bot$ \\
		$\tau' (\top)$ &  = & $\top$ && $\sigma' (\top)$ &  = & $\top$ \\
		$\tau' (\phi\wedge \psi)$ &  = & $\tau' (\phi)\wedge \tau'(\psi)$ && $\sigma (\phi\wedge \psi)$ &  = & $\sigma' (\phi)\wedge \sigma'(\psi)$  \\
		$\tau' (\phi\vee \psi)$ &  = & $\tau' (\phi)\vee \tau'(\psi)$ && $\sigma' (\phi\vee \psi)$ &  = & $\sigma'(\phi)\vee \sigma'(\psi)$  \\
		$\tau' (\phi \rightarrow \psi)$ &  = & $\Box_{\leq}(\neg \tau' (\phi)\vee \tau'(\psi))$ && $\sigma' (\phi \rightarrow \psi)$ &  = & $\Box_{\leq}(\neg \sigma' (\phi)\vee \sigma'(\psi))$.\\
		$\tau' (\phi \pdra \psi)$ &  = & $\Diamond_{\geq}(\neg \tau' (\phi) \wedge \tau'(\psi))$ &&$\sigma' (\phi \pdra \psi)$ &  = & $\Diamond_{\geq}(\neg\sigma' (\phi) \wedge\sigma'(\psi))$.  \\
	\end{tabular}
\end{center}
Notice that $\tau'$ and $\sigma'$ agree on each defining clause but those relative to the proposition variables.
Let $\bbA$ be a  bi-Heyting algebra  and $\bbB$  a Boolean algebra  such that $e\colon \bbA\hookrightarrow\bbB$ is an order-embedding and a homomorphism of the lattice reducts of $\bbA$ and $\bbB$. Suppose that $e$ has both a left adjoint $c\colon \bbB\to \bbA$ and a right adjoint $\iota\colon \bbB\to \bbA$ such that identities \eqref{eq: rightarrow Rightarrow e and iota} and \eqref{eq: pdra minus e and c} hold for every $a, b\in \bbA$. Then $\bbB$ can be endowed with a natural structure of bi-modal S4-algebra by defining $\Box^\bbB\colon \bbB\to \bbB$ by the assignment $b\mapsto (e\circ \iota)(b)$   and $\Diamond^\bbB\colon \bbB\to \bbB$ by the assignment $b\mapsto (e\circ c)(b)$.
\begin{prop}
	The BAE $(\bbB, \Box^\bbB, \Diamond^\bbB)$, with $\Box^\bbB, \Diamond^\bbB$ defined as above, is normal and  an S4-bimodal algebra.
\end{prop}
The following proposition show that Proposition \ref{prop: main prop of Godel tarski algebraically} applies to $\tau'$ and $\sigma'$. We let $X:= \mathsf{Prop}$. The proof is similar to those of Propositions \ref{prop:tau verifies a and b}  and \ref{prop:sigma verifies a and b}, and is omitted.
\begin{prop}
	\label{prop: sigma prime tau prime satisfy a and b}
	The translation $\tau'$ (resp.\ $\sigma'$) defined above satisfies conditions (a) and (b)  of Proposition \ref{prop: main prop of Godel tarski algebraically} relative to $r: \bbB^X\to \bbA^X$ defined by $U\mapsto (\iota\circ U)$ (resp.\ defined by $U\mapsto (c\circ U)$).
\end{prop}

Thanks to the proposition above,  Proposition \ref{prop: main prop of Godel tarski algebraically} applies to both $\tau'$ and $\sigma'$ provided suitable embeddings of bi-Heyting algebras into Boolean algebras exist. The existence of such  embeddings is proven by Rauszer in \cite[Section 4]{rauszer1974semi}. We now give a reformulation of this result in the format required to instantiate Proposition \ref{prop: main prop of Godel tarski algebraically} and highlighting the compatibility of this embedding with canonical extensions --- as well as a duality-based proof.
\begin{prop}\label{prop:existence of algebras and adjoints:biHeyting}
	For every bi-Heyting algebra $\bbA$, there exists a Boolean algebra $\bbB$ such that $\bbA$ embeds into $\bbB$ via some order-embedding $e\colon \bbA\hookrightarrow\bbB$ which is also a homomorphism of the lattice reducts of $\bbA$ and $\bbB$ and has both a left adjoint $c\colon \bbB\to \bbA$ and a right adjoint $\iota\colon \bbB\to \bbA$ verifying conditions \eqref{eq: rightarrow Rightarrow e and iota} and \eqref{eq: pdra minus e and c}. Finally, all these facts lift to the canonical extensions of $\bbA$ and $\bbB$ as in the following diagram:
	
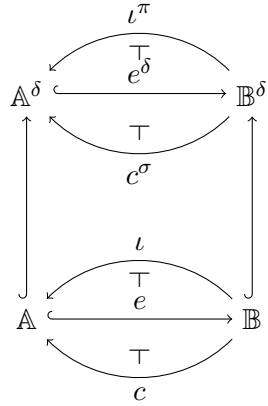
\begin{center}
		\begin{tikzpicture}[node/.style={circle, draw, fill=black}, scale=1]
		
		\node (A) at (-1.5,-1.5) {$\mathbb{A}$};
		\node (A delta) at (-1.5,1.5) {$\mathbb{A}^{\delta}$};
		\node (B) at (1.5,-1.5) {$\mathbb{B}$};
		\node (B delta) at (1.5,1.5) {$\mathbb{B}^{\delta}$};
		
		\draw [right hook->] (A) to  (A delta);
		\draw [right hook->] (B)  to (B delta);
		\draw [right hook->] (A)  to node[above] {$e$}  (B);
		\draw [right hook->] (A delta)  to node[above] {$e^{\delta}$}  (B delta);
		
		\draw [->] (B delta) to [out=135,in=45, looseness=1]  node[below] {\rotatebox[origin=c]{270}{$\vdash$}} node[above] {$\iota^{\pi}$}  (A delta);
		\draw [->] (B) to [out=135,in=45, looseness=1]  node[below] {\rotatebox[origin=c]{270}{$\vdash$}} node[above] {$\iota$}  (A);
		\draw [->] (B delta) to [out=225,in=-45, looseness=1]  node[above] {\rotatebox[origin=c]{270}{$\vdash$}} node[below] {$c^{\sigma}$}  (A delta);
		\draw [->] (B) to [out=225,in=-45, looseness=1]  node[above] {\rotatebox[origin=c]{270}{$\vdash$}} node[below] {$c$}  (A);
		\end{tikzpicture}
	\end{center}
\end{prop}
\begin{proof}
	Via  Esakia-type duality \cite{Es75, wolter1998CoImplication}, the bi-Heyting algebra $\bbA$ can be identified  with the algebra of clopen up-sets of its associated dual space $\mathbb{X}_\bbA$  (referred to here as a a bi-Esakia space), which is a Priestley space, hence a Stone space. Let $\bbB$ be the Boolean algebra of the clopen sub{\em sets} of $\mathbb{X}_\bbA$. Since any clopen up-set is in particular a clopen subset,  a natural order embedding $e: \bbA\hookrightarrow\bbB$ exists, which is also a lattice homomorphism between $\bbA$ and $\bbB$. This shows the first part of the claim. As to the second part, bi-Esakia spaces are Priestley spaces such that both  the upward-closure and the downward-closure of a clopen set is a clopen set.
	
	Therefore, we can define the maps  $c\colon \bbB\to \bbA$  and $\iota\colon \bbB\to \bbA$ by the assignments $b\mapsto  b{\uparrow}$ and $b\mapsto  \neg((\neg b){\downarrow})$ respectively, where $b$ is identified with its corresponding clopen set in $\mathbb{X}_\bbA$, $\neg b$ is defined as the relative complement of $b$ in $\mathbb{X}_\bbA$, and  $b{\uparrow}$ and $(\neg b){\downarrow}$ are defined using the order in $\mathbb{X}_\bbA$. It can be readily verified that $c$ and $\iota$ are the left and right adjoints of $e$ respectively, and that moreover  conditions \eqref{eq: rightarrow Rightarrow e and iota} and \eqref{eq: pdra minus e and c} hold. 
	
	
	Finally, $e:\bbA\rightarrow\bbB$ being also a homomorphism between the lattice reducts of $\bbA$ and $\bbB$ implies that $e$ is smooth and its canonical extension $e^\delta\colon \bbad\to \bbbd$, besides being an order-embedding, is a complete homomorphism between the lattice reducts of $\bbad$ and $\bbbd$ (cf.\  \cite[Corollary 4.8]{GeHa01}), and hence is endowed with both a left and a right adjoint. Furthermore, the left (resp.\ right) adjoint of $e^\delta$ coincides with $c^\sigma$ (resp.\ with $\iota^\pi$) (cf.\  \cite[Proposition 4.2]{GePr07}). Hence, $\bbbd$ can be naturally endowed with the structure of an S4 bi-modal algebra by defining $\Box_{\leq}^{\bbbd}\colon \bbbd\to \bbbd$ by the assignment $u\mapsto (e^\delta\circ \iota^\pi)(u)$, and $\Diamond_{\geq}^{\bbbd}\colon \bbbd\to \bbbd$ by the assignment $u\mapsto (e^\delta\circ c^\sigma)(u)$.
\end{proof}

\subsection{Parametric GMT translations}\label{Sec:ParamGMTTrans}

In this section we extend the previous translations to a parametric set of GMT translations for each normal DLE- and bHAE-logic (cf.~Section \ref{sec:preliminaries}). 
Parametric GMT translations were already considered in \cite{GeNaVe05} in the context of one specific DLE signature, namely that of Distributive Modal Logic, where they are used to prove the transfer of correspondence results for $\epsilon$-Sahlqvist inequalities in every order-type $\epsilon$ (cf.~Definition \ref{Inducive:Ineq:Def}). We generalize this idea to arbitrary DLE-logics, and explore the additional properties of GMT translations in the setting of bHAE-logics.


\subsubsection{Parametric GMT translations for normal DLE-logics}\label{sssection:parametric DLE}
In the present section we consider parametric GMT translations for the general DLE-setting. We will use the following notation: for every Boolean algebra $\bbB$, $n$-tuple $\overline{b}$ of elements of $\bbB$ and every order-type $\eta$ on $n$, we let $\overline{b}^\eta: = (b'_i)_{i = 1}^n$ where $b'_i = b_i$ if $\eta(i) = 1$ and $b'_i = \neg b_i$ if $\eta(i) = \partial$. Let us fix a normal DLE-signature $\mathcal{L}_{\mathrm{DLE}} = \mathcal{L}_{\mathrm{DLE}}(\mathcal{F}, \mathcal{G})$. We first identify the target language for these translations.  This is the  normal BAE-signature $\mathcal{L}^{\circ}_{\mathrm{BAE}} = \mathcal{L}_{\mathrm{BAE}}(\mathcal{F}^{\circ}, \mathcal{G}^{\circ})$ associated with $\mathcal{L}_{\mathrm{DLE}}$, where $\mathcal{F}^{\circ}: = \{\Diamond_\geq\}\cup
\{f^{\circ}\mid f\in \mathcal{F}\}$, and $\mathcal{G}^{\circ}: = \{\Box_\leq\}\cup
\{g^{\circ}\mid g\in \mathcal{G}\}$, and for every $f\in \mathcal{F}$ (resp.\ $g\in \mathcal{G}$), the  connective $f^{\circ}$ (resp.\ $g^{\circ}$) is such that $n_{f^{\circ}} = n_{f}$ (resp.\ $n_{g^{\circ}} = n_{g}$) and $\epsilon_{f^{\circ}}(i) = 1$ for each $1\leq i\leq n_f$ (resp.\ $\epsilon_{g^{\circ}}(i) = 1$ for each $1\leq i\leq n_g$).

We assume that an order-embedding $e\colon \bbA\hookrightarrow\bbB$ exists, which is a homomorphism of the lattice reducts of  $\bbA$ and $\bbB$, such that  both the left and right adjoint $c \colon \bbB\to\bbA$ and $\iota\colon \bbB\to \bbA$ exist
%
and moreover the following diagrams commute for every $f\in \mathcal{F}$ and $g\in \mathcal{G}$:\footnote{Notice that equations \eqref{eq: rightarrow Rightarrow e and iota} and \eqref{eq: pdra minus e and c} encode the special cases of the commutativity of the diagrams \eqref{eq: e and c and iota and f and g} for $f(\phi, \psi): = \phi\pdra\psi$ (in which case, $f^\circ(\neg\alpha, \beta): = \neg\alpha\wedge \beta$) and $g(\phi, \psi): = \phi\rightarrow\psi$ (in which case, $g^\circ(\neg \alpha, \beta): = \neg \alpha\vee \beta$).}
\begin{equation}\label{eq: e and c and iota and f and g}
\begin{array}{ccc}

\begin{CD}
\bbA^{n_f} @>e^{\epsilon_f}>> \bbB^{n_f}\\
@VVf^{\bbA}V @VV{f^\circ}^{\bbB}V\\
\bbA @<c<< \bbB
\end{CD}
&\quad&

\begin{CD}
\bbA^{n_g} @>e^{\epsilon_g}>> \bbB^{n_g}\\
@VVg^{\bbA}V @VV{g^\circ}^{\bbB}V\\
\bbA @<\iota<< \bbB
\end{CD}
\\
\end{array}
\end{equation}
where $e^{\epsilon_f}(\overline{a}): = \overline{e(a)}^{\epsilon_f}$ and $e^{\epsilon_g}(\overline{a}): = \overline{e(a)}^{\epsilon_g}$.
Then, as discussed early on, the Boolean reduct of $\bbB$ can be endowed with a natural structure of bi-modal S4-algebra by defining $\Box^\bbB\colon \bbB\to \bbB$ by the assignment $b\mapsto (e\circ \iota)(b)$   and $\Diamond^\bbB\colon \bbB\to \bbB$ by the assignment $b\mapsto (e\circ c)(b)$.

The target language  for the parametrized GMT translations over $\mathsf{Prop}$ is given by
\[\mathcal{L}^\circ_{BAE}\ni\alpha:: = p\mid \bot \mid \alpha\vee \alpha\mid\alpha\wedge \alpha \mid\neg\alpha\mid f^\circ(\overline{\alpha}) \mid g^\circ(\overline{\alpha})\mid \Diamond_{\geq}\alpha\mid \Box_{\leq}\alpha.\]

Let $X: = \mathsf{Prop}$. For any order-type $\epsilon$ on $X$, define the translation $\tau_\epsilon : \mathcal{L}_{\mathrm{DLE}}\to \mathcal{L}^\circ_{\mathrm{BAE}}$ by the following recursion:

\begin{center}
	\begin{tabular}{c c c}
		$
		\tau_\epsilon(p) = \begin{cases} \Box_{\leq}  p &\mbox{if } \epsilon(p) = 1 \\
		\Diamond_{\geq} p & \mbox{if } \epsilon(p) = \partial, \end{cases}
		$
		& $\quad$ &

		\begin{tabular}{r c l}
			$\tau_\epsilon (\bot)$ &  = & $\bot$ \\
			$\tau_\epsilon (\top)$ &  = & $\top$ \\
			$\tau_\epsilon (\phi\wedge \psi)$ &  = & $\tau_\epsilon (\phi)\wedge \tau_\epsilon(\psi)$  \\
			$\tau_\epsilon (\phi\vee \psi)$ &  = & $\tau_\epsilon (\phi)\vee \tau_\epsilon(\psi)$  \\
			$\tau_\epsilon (f(\overline{\phi}))$ &  = & $\Diamond_{\geq}f^\circ(\overline{\tau_\epsilon (\phi)}^{\epsilon_f})$  \\
			$\tau_\epsilon (g(\overline{\phi}))$ &  = & $\Box_{\leq}g^\circ(\overline{\tau_\epsilon (\phi)}^{\epsilon_g})$  \\
		\end{tabular}
		\\
	\end{tabular}
\end{center}
where for each order-type $\eta$ on $n$ and any $n$-tuple $\overline{\psi}$ of $\mathcal{L}^\circ_{\mathrm{BAE}}$-formulas,  $\overline{\psi}^\eta$ denotes the $n$-tuple $(\psi'_i)_{i = 1}^n$, where $\psi_i' = \psi_i$  if $\eta(i) = 1$ and $\psi_i' = \neg \psi_i$  if $\eta(i) = \partial$.\label{notation:polarity bookkeeping}

Let $\bbA$ be a $\mathcal{L}_{\mathrm{DLE}}$-algebra and $\bbB$ be a  $\mathcal{L}^\circ_{\mathrm{BAE}}$-algebra such that an order-embedding $e\colon \bbA\hookrightarrow\bbB$ exists, which is a homomorphism of the lattice-reducts of $\bbA$ and $\bbB$, is endowed with  both  right and left adjoints, and satisfies the commutativity of the diagrams \eqref{eq: e and c and iota and f and g}  for every $f\in \mathcal{F}$ and $g\in \mathcal{G}$. For every order-type $\epsilon$ on $X$, consider the map $r_\epsilon: \bbB^X\to \bbA^X$ defined, for any $U\in \bbB^X$ and $p\in X$, by:
\[
r_\epsilon (U)(p) = \begin{cases}
(\iota\circ U) (p) & \mbox{if } \epsilon(p) = 1\\
(c\circ U) (p) & \mbox{if } \epsilon(p) = \partial\\
\end{cases}
\]
\begin{prop}
	\label{prop:  tau epsilon for DLE satisfies a and b}
	For every order-type $\epsilon$ on $X$, the translation $\tau_\epsilon$  defined above satisfies conditions (a) and (b)  of Proposition \ref{prop: main prop of Godel tarski algebraically} relative to $r_\epsilon$.
\end{prop}

\begin{proof}
	By induction on $\phi$. As for the base case, let $\phi: = p\in \mathsf{Prop}$. If $\epsilon(p) = \partial$, then  for any $U\in \bbB^X$ and   $V\in\bbA^X$,
	\begin{center}
		\begin{tabular}{r c l l r c l l}
			$e(\val{p}_{r_\epsilon(U)})$ & = & $e((c\circ U)(p))$ &(def.\ of $r_\epsilon$) &$\valb{\tau_\epsilon(p)}_{\overline{e}(V)}$ & = & $\valb{\Diamond_{\geq} p}_{\overline{e}(V)}$ & (def.\ of $\tau_\epsilon$) \\
			& = & $(e\circ c)\valb{p}_U$    &(assoc.\ of $\circ$) && = & $\Diamond^\bbB\valb{p}_{\overline{e}(V)}$ & (def.\ of $\valb{\cdot}_U$)\\
			& = & $\Diamond^\bbB\valb{p}_U$ & (def.\ of $\Diamond^\bbB$) && = & $\Diamond^\bbB((e\circ V)(p))$ & (def.\ of $\overline{e}(V)$) \\
			& = & $\valb{\Diamond_{\geq} p}_U$ & (def.\ of $\valb{\cdot}_U$) && = & $(e\circ c)((e\circ V)(p))$ & (def.\ of $\Diamond^\bbB$) \\
			& = & $\valb{\tau_\epsilon( p)}_U$. &(def.\ of $\tau_\epsilon$) && = & $e((c \circ e)(V(p)))$ & (assoc.\ of $\circ$)\\
			&&&&& = & $e(V(p))$ & ($e\circ(c\circ e) = e$)\\
			&&&&& = & $e(\val{p}_V)$. & (def.\ of $\val{\cdot}_V$) \\
		\end{tabular}
	\end{center}
	If  $\epsilon(p) = 1$, then for any $U\in \bbB^X$ and   $V\in\bbA^X$,
	\begin{center}
		\begin{tabular}{r c l l r c l l}
			$e(\val{p}_{r_\epsilon(U)})$ & = & $e((\iota\circ U)(p))$ &(def.\ of $r_\epsilon$) &$\valb{\tau_\epsilon(p)}_{\overline{e}(V)}$ & = & $\valb{\Box_{\leq} p}_{\overline{e}(V)}$ & (def.\ of $\tau_\epsilon$) \\
			& = & $(e\circ \iota)\valb{p}_U$    &(assoc.\ of $\circ$) && = & $\Box^\bbB\valb{p}_{\overline{e}(V)}$ & (def.\ of $\valb{\cdot}_U$)\\
			& = & $\Box^\bbB\valb{p}_U$ & (def.\ of $\Box^\bbB$) && = & $\Box^\bbB((e\circ V)(p))$ & (def.\ of $\overline{e}(V)$) \\
			& = & $\valb{\Box_{\leq} p}_U$ & (def.\ of $\valb{\cdot}_U$) && = & $(e\circ \iota)((e\circ V)(p))$ & (def.\ of $\Box^\bbB$) \\
			& = & $\valb{\tau_\epsilon( p)}_U$. &(def.\ of $\tau_\epsilon$) && = & $e((\iota \circ e)(V(p)))$ & (assoc.\ of $\circ$)\\
			&&&&& = & $e(V(p))$ & ($e\circ(\iota\circ e) = e$)\\
			&&&&& = & $e(\val{p}_V)$. & (def.\ of $\val{\cdot}_V$) \\
		\end{tabular}
	\end{center}
	Let $\phi: = f(\overline{\phi})$. Then  for any $U\in \bbB^X$ and   $V\in\bbA^X$,
	\begin{center}
		\begin{tabular}{r c l l r c l l}
			&   &$e(\val{f(\overline{\phi})}_{r_\epsilon(U)})$  &                                         & &   & $\valb{\tau_\epsilon(f(\overline{\phi}))}_{\overline{e}(V)}$ \\
			& = & $e(f(\overline{\val{\phi}_{r_\epsilon(U)}}))$ &(def.\ of $\val{\cdot}_{r_\epsilon(U)}$) & & = & $\valb{\Diamond_{\geq}f^\circ(\overline{\tau_\epsilon(\phi)}^{\epsilon_f})}_{\overline{e}(V)}$ & (def.\ of $\tau_\epsilon$) \\
			& = & $e(c\circ f^\circ(\overline{e(\val{\phi}_{r_\epsilon(U)})}^{\epsilon_f})$    &(assump.\ \eqref{eq: e and c and iota and f and g}) && = & $\Diamond^{\mathbb{B}}f^\circ(\overline{\valb{\tau_\epsilon(\phi)}_{\overline{e}(V)}}^{\epsilon_f})$ & (def.\ of $\valb{\cdot}_{\overline{e}(V)}$)\\
			& = & $\Diamond^{\mathbb{B}}f^\circ(\overline{\valb{\tau_\epsilon(\phi)}_{U}}^{\epsilon_f})$  & (IH \& def.~ of $\Diamond^{\mathbb{B}}$) && = & $\Diamond^{\mathbb{B}}f^\circ(\overline{e(\val{\phi}_V)}^{\epsilon_f})$ & (IH) \\
			& = & $\valb{\Diamond_{\geq}f^\circ(\overline{\tau_\epsilon(\phi)}^{\epsilon_f})}_U$ & (def.\ of $\valb{\cdot}_U$) && = &
			$e (c \circ f^\circ(\overline{e(\val{\phi}_V)}^{\epsilon_f}))$ & (def.\ of $\Diamond^{\mathbb{B}}$)\\
			& = & $\valb{\tau_\epsilon(f(\overline{\phi}))}_U$. &(def.\ of $\tau_\epsilon$) && = &
			$e(f(\overline{\val{\phi}_V}))$ & (assump.\ \eqref{eq: e and c and iota and f and g}) \\
			& & & && = &
			$e(\val{f(\overline{\phi})}_V)$ & (def.\ of $\val{\cdot}_V$)\\
		\end{tabular}
	\end{center}
	
	The remaining cases are analogous and are omitted.
\end{proof}
As a consequence of the proposition above, Proposition \ref{prop: main prop of Godel tarski algebraically} applies to  $\tau_\epsilon$ for any order-type $\epsilon$ on $X$. Hence:
\begin{cor}\label{cor:main theorem for tau epsilon DLE}
	Let $\bbA$ be a $\mathcal{L}_{\mathrm{DLE}}$-algebra. If an embedding $e:\bbA\hookrightarrow\bbB$ exists into a  $\mathcal{L}^\circ_{\mathrm{BAE}}$-algebra $\bbB$ which is a homomorphism of the lattice reducts of $\bbA$ and $\bbB$, and $e$ has both a right adjoint $\iota\colon \bbB\to \bbA$ and a left adjoint $c\colon \bbB\to \bbA$ satisfying the commutativity of the  diagrams \eqref{eq: e and c and iota and f and g} for every $f\in \mathcal{F}$ and $g\in \mathcal{G}$, then for any $\mathcal{L}_{\mathrm{DLE}}$-inequality $\phi \leq \psi$,
	\[
	\bbA\models \phi\leq \psi \quad \mbox{ iff } \quad \bbB \models\tau_{\epsilon}(\phi)\leq \tau_{\epsilon}(\psi).
	\]
\end{cor}

We finish this subsection by showing that every {\em perfect} $\mathcal{L}_{\mathrm{DLE}}$-algebra $\bbA$ (cf.~Definition \ref{def:DLE}) embeds into a {\em perfect} $\mathcal{L}^\circ_{\mathrm{BAE}}$-algebra $\bbB$ in the way described in Corollary \ref{cor:main theorem for tau epsilon DLE}:

\begin{prop}\label{prop:existence of algebras and adjoints:perfectDLEs}
	For every perfect $\mathcal{L}_{\mathrm{DLE}}$-algebra $\bbA$, there exists a perfect  $\mathcal{L}^\circ_{\mathrm{BAE}}$--algebra $\bbB$ such that $\bbA$ embeds into $\bbB$ via some order-embedding $e\colon \bbA\hookrightarrow\bbB$ which is also a homomorphism of the lattice reducts of $\bbA$ and $\bbB$ and has both a left adjoint $c\colon \bbB\to \bbA$ and a right adjoint $\iota\colon \bbB\to \bbA$ satisfying the commutativity of the  diagrams \eqref{eq: e and c and iota and f and g}.
\end{prop}
\begin{proof}
	Via expanded Birkhoff's  duality (cf.~e.g.~\cite{DaPr90, sofronie2000duality})  the  perfect $\mathcal{L}_{\mathrm{DLE}}$-algebra $\bbA$ can be identified  with the algebra of up-sets of its associated prime element $\mathcal{L}_{\mathrm{DLE}}$-frame $\mathbb{X}_\bbA$, which is based on a poset. Let $\bbB$ be the powerset algebra of the universe of $\mathbb{X}_\bbA$. Since any  up-set is in particular a subset,  a natural order embedding $e: \bbA\hookrightarrow\bbB$ exists, which is also a complete lattice homomorphism between $\bbA$ and $\bbB$. This shows the first part of the claim.
	
	As to the second part, because $e$ is a complete homomorphism between complete lattices, it has both a left adjoint $c\colon \bbB\to \bbA$ and a right adjoint  $\iota\colon \bbB\to \bbA$, respectively defined by the assignments $b\mapsto  b{\uparrow}$ and $b\mapsto  \neg((\neg b){\downarrow})$, where $b$ is identified with its corresponding  subset in $\mathbb{X}_\bbA$, $\neg b$ is defined as the relative complement of $b$ in $\mathbb{X}_\bbA$, and  $b{\uparrow}$ and $(\neg b){\downarrow}$ are defined using the order in $\mathbb{X}_\bbA$. 
	
	Finally, notice that any $\mathcal{L}_{\mathrm{DLE}}$-frame $\mathbb{F}$ is also an $\mathcal{L}^\circ_{\mathrm{BAE}}$-frame by interpreting the $f$-type connective $\Diamond_{\geq}$ by means of the binary relation  $\geq$, the $g$-type connective $\Box_{\leq}$ by means of the binary relation $\leq$,  each $f^\circ\in \mathcal{F}^\circ$ by means of $R_f$ and each $g^\circ\in \mathcal{G}^\circ$ by means of $R_g$. Moreover, the additional properties \eqref{eq:compatibility Rf} and \eqref{eq:compatibility Rg} of the relations $R_f$ and $R_g$ guarantee that the  diagrams \eqref{eq: e and c and iota and f and g} commute for every $f\in \mathcal{F}$ and $g\in \mathcal{G}$. 
	%
\end{proof}
\begin{remark}
	\label{rmk: mix sufficient but not necessary}
	The parametric GMT translations defined in this section do not just  generalize  those of \cite{GeNaVe05} w.r.t.~the signature, but also differ from them in terms of their definition, and the assumptions each of which  requires. Specifically,  the parametric translations of $f$-formulas (resp.~$g$-formulas)   add an extra $\Diamond_{\geq}$ (resp.~a $\Box_{\leq}$) on top of the corresponding $f^\circ$ (resp.~$g^\circ$) connective, while in \cite{GeNaVe05}, the extra $\Diamond_{\geq}$ and $\Box_{\leq}$ are not used in the definition of the translation of formulas with a modal operator as main connective. This simpler definition is sound  only w.r.t.~semantic settings, such as that of \cite{GeNaVe05}, in which the relations interpreting the modal connectives satisfy the additional properties \eqref{eq:compatibility Rf} and \eqref{eq:compatibility Rg}. This corresponds algebraically to the operations interpreting the classical modal connectives restricting nicely to the algebra of targets of persistent valuations, and corresponds syntactically to the {\em mix axioms} (e.g.~$\Box_{\leq}\Box^\circ\Box_{\leq} p\leftrightarrow \Box^\circ p$ and $\Diamond_{\geq}\Diamond^\circ\Diamond_{\geq} p\leftrightarrow \Diamond^\circ p$, cf.~\cite[Section 6]{wolter1998CoImplication}) being valid. The particular algebras and maps picked in the proof of Proposition \ref{prop:existence of algebras and adjoints:perfectDLEs}  happen to validate also the mix axioms. However, the mix axioms are not a necessary condition for satisfying Proposition \ref{prop:existence of algebras and adjoints:perfectDLEs}, 
	as the following example shows.
	
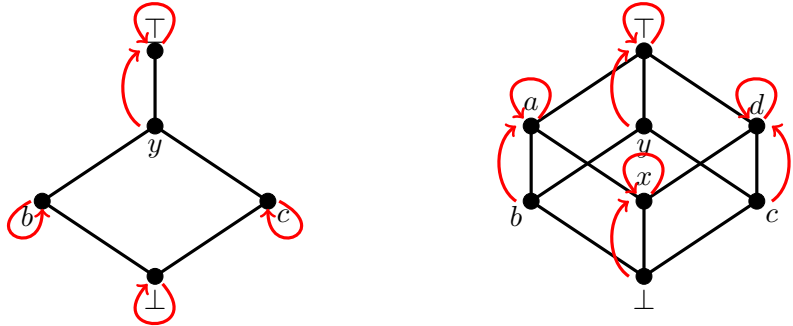
\begin{center}
		\begin{tikzpicture}
		\filldraw[black] (1,5.5) circle (3 pt) node {$\top$}; 
		\filldraw[black] (1,4.5) circle (3 pt);
		\filldraw[black] (-0.5,3.5) circle (3 pt);
		\filldraw[black] (2.5,3.5) circle (3 pt);
		\filldraw[black] (1,2.5) circle (3 pt);
		\draw[very thick] (1, 5.5) -- (1, 4.5) --
		(-0.5, 3.5) -- (1, 2.5) -- (2.5, 3.5) -- (1, 4.5);
		\draw[very thick, red, <-] (0.9, 2.4)  .. controls (0.3, 1.7) and  (1.7, 1.7)  .. (1.1, 2.4); 
		\draw[very thick, red, ->] (-0.65, 3.5)  .. controls (-1.4, 3.1) and  (-0.5, 2.7)  .. (-0.5, 3.4); 
		\draw[very thick, red, <-] (0.9, 5.6)  .. controls (0.3, 6.3) and  (1.7, 6.3)  .. (1.1, 5.6); 
		\draw[very thick, red, ->] (2.65, 3.5)  .. controls (3.4, 3.1) and  (2.5, 2.7)  .. (2.5, 3.4); 
		\draw[very thick, red, ->] (0.8, 4.5)  .. controls (0.5, 4.7) and  (0.5, 5.3)  .. (0.8, 5.5); 
		\draw (1, 2.2) node {$\bot$};
		\draw (-0.7, 3.3) node {$b$};
		\draw (1, 5.8) node {$\top$};
		\draw (2.7, 3.3) node {$c$};
		\draw (1, 4.2) node {$y$};

		\draw[very thick] (7.5, 2.5) -- (6, 3.5) --
		(6, 4.5) -- (7.5, 5.5) -- (9, 4.5) -- (9, 3.5) -- (7.5, 4.5) -- (6, 3.5);
		\draw[very thick] (7.5, 5.5) -- (7.5, 4.5);
		\draw[very thick] (7.5, 2.5) -- (9, 3.5);
		\draw[very thick] (7.5, 2.5) -- (7.5, 3.5);
		\draw[very thick] (6, 4.5) -- (7.5, 3.5) -- (9, 4.5);
		\filldraw[black] (7.5,2.5) circle (3 pt); 
		\filldraw[black] (6,3.5) circle (3 pt); 
		\filldraw[black] (6,4.5) circle (3 pt); 
		\filldraw[black] (7.5,5.5) circle (3 pt); 
		\filldraw[black] (9,4.5) circle (3 pt); 
		\filldraw[black] (7.5,3.5) circle (3 pt); 
		\filldraw[black] (9,3.5) circle (3 pt); 
		\filldraw[black] (7.5,4.5) circle (3 pt); 
		\draw (7.5, 2.2) node {$\bot$};
		\draw (5.8, 3.3) node {$b$};
		\draw (6, 4.8) node {$a$};
		\draw (7.5, 5.8) node {$\top$};
		\draw (9, 4.8) node {$d$};
		\draw (9.2, 3.3) node {$c$};
		\draw (7.5, 4.2) node {$y$};
		\draw (7.5, 3.8) node {$x$};
		
		\draw[very thick, red, ->] (5.8, 3.5)  .. controls (5.5, 3.7) and  (5.5, 4.3)  .. (5.8, 4.5); 
		\draw[very thick, red, ->] (9.2, 3.5)  .. controls (9.5, 3.7) and  (9.5, 4.3)  .. (9.2, 4.5); 
		\draw[very thick, red, <-] (7.4, 5.6)  .. controls (6.8, 6.3) and  (8.2, 6.3)  .. (7.6, 5.6); 
		\draw[very thick, red, ->] (7.3, 4.5)  .. controls (7, 4.7) and  (7, 5.3)  .. (7.3, 5.5); 
		\draw[very thick, red, ->] (7.3, 2.5)  .. controls (7, 2.7) and  (7, 3.3)  .. (7.3, 3.5); 
		\draw[very thick, red, <-] (7.4, 3.6)  .. controls (6.8, 4.3) and  (8.2, 4.3)  .. (7.6, 3.6); 
		\draw[very thick, red, <-] (5.9, 4.6)  .. controls (5.3, 5.3) and  (6.7, 5.3)  .. (6.1, 4.6); 
		\draw[very thick, red, <-] (8.9, 4.6)  .. controls (8.3, 5.3) and  (9.7, 5.3)  .. (9.1, 4.6); 
		
		\end{tikzpicture}
	\end{center}
	
	The (finite, hence perfect) distributive lattice with $\Box$ on the left embeds as a complete lattice into the (finite, hence perfect) Boolean algebra with $\Box^\circ$ on the right. Hence, the corresponding embedding $e$ has  both a right adjoint $\iota$ and a left adjoint $c$. The composition $e\circ\iota$ gives rise to the S4-operation $\Box_{\leq}: = e\circ\iota$ which by construction  maps $a$ to $b$, $d$ to $c$, $x$ to $\bot$, and every other element to itself.   It it easy to check that $\Box$ and $\Box^\circ$ verify the commutativity of diagram \eqref{eq: e and c and iota and f and g}. However, $\Box_{\leq}\Box^\circ\Box_{\leq} b  = b\neq a  = \Box^\circ b$, which shows that the mix axiom is not valid.
\end{remark}
\noindent Notice that Proposition \ref{prop:existence of algebras and adjoints:perfectDLEs} has a more restricted scope than analogous propositions such as Propositions \ref{prop:existence of algebras and adjoints:biHeyting} or \ref{prop:existence of algebras and adjoints:Heyting}. Indeed, via expanded Priestley duality (cf.~e.g.~\cite{sofronie2000duality}), any  DLE $\bbA$ is isomorphic to the DLE of clopen up-sets of its dual (relational) Priestley space $\mathbb{X}_\bbA$, which is a Stone space in particular, and this yields a natural embedding of $\bbA$ into the BAE of the clopen subsets of $\mathbb{X}_\bbA$. However, this embedding  has in general neither a right nor a left adjoint. In Section \ref{sec:Corresp:Via:Trans}, we will see  that Proposition \ref{prop:existence of algebras and adjoints:perfectDLEs} is enough to obtain the  correspondence theorem for inductive $\mathcal{L}_{\mathrm{DLE}}$-inequalities via translation from the correspondence theorem for inductive $\mathcal{L}_{\mathrm{BAE}}$-inequalities. However, we will see in Section \ref{sec:canonicity} that canonicity cannot be straightforwardly obtained in the same way, precisely due to the restriction on Proposition \ref{prop:existence of algebras and adjoints:perfectDLEs}.
As we show next, this restriction can be removed if we confine ourselves setting to arbitrary normal bHAEs. In this setting, we are going to show a strengthened version of Proposition \ref{prop:existence of algebras and adjoints:perfectDLEs}  which will be key for the transfer of canonicity of Section \ref{ssec:canonicity bHAE}.

\subsubsection{Parametric GMT translations for bHAE-logics}\label{sec:ParamGMTbHAE}

The considerations collected in Section \ref{sssection:parametric DLE} apply to the more restricted setting of bHAEs (cf.~Section \ref{ssec: DLEs}). Let us fix a bHAE-signature $\mathcal{L}_{\mathrm{bHAE}} = \mathcal{L}_{\mathrm{bHAE}}(\mathcal{F}, \mathcal{G})$ and let $\mathcal{L}^\circ_{\mathrm{BAE}}$ denote its corresponding target signature (cf.~Section \ref{sssection:parametric DLE}). Then, Corollary \ref{cor:main theorem for tau epsilon DLE} specializes as follows:

\begin{cor}\label{cor:main theorem for tau epsilon bi}
	Let $\bbA$ be an $\mathcal{L}_{\mathrm{bHAE}}$-algebra. If an embedding $e:\bbA\rightarrow\bbB$ exists into an  $\mathcal{L}^\circ_{\mathrm{BAE}}$-algebra $\bbB$ which is a homomorphism of their lattice reducts and $e$ has both a right adjoint $\iota\colon \bbB\to \bbA$ and a left adjoint $c\colon \bbB\to \bbA$ satisfying \eqref{eq: rightarrow Rightarrow e and iota}, \eqref{eq: pdra minus e and c} and \eqref{eq: e and c and iota and f and g}, then for any $\mathcal{L}_{\mathrm{bHAE}}$-inequality $\phi \leq \psi$,
	
	\[
	\bbA\models \phi\leq \psi \quad \mbox{ iff } \quad \bbB \models\tau_{\epsilon}(\phi)\leq \tau_{\epsilon}(\psi).
	\]
\end{cor}

As discussed above, the present setting is characterized  by the fact that, for any $\mathcal{L}_{\mathrm{bHAE}}$-algebras $\bbA$, the right and left adjoints of the embedding map $e: \bbA \hookrightarrow\bbB$ exist, as shown by the following proposition.
\begin{prop}\label{prop:existence of algebras and adjoints:bHAE}
	For every  $\mathcal{L}_{\mathrm{bHAE}}$-algebra $\bbA$, there exists an   $\mathcal{L}^\circ_{\mathrm{BAE}}$-algebra $\bbB$ such that $\bbA$ embeds into $\bbB$ via some order-embedding $e\colon \bbA\hookrightarrow\bbB$ which is also a homomorphism of the lattice reducts of $\bbA$ and $\bbB$ and has both a left adjoint $c\colon \bbB\to \bbA$ and a right adjoint $\iota\colon \bbB\to \bbA$ satisfying \eqref{eq: rightarrow Rightarrow e and iota}, \eqref{eq: pdra minus e and c} and \eqref{eq: e and c and iota and f and g}.
	Finally, all these facts lift to the canonical extensions of $\bbA$ and $\bbB$ as in the following diagram:
	
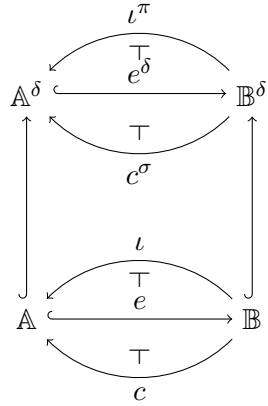
\begin{center}
		\begin{tikzpicture}[node/.style={circle, draw, fill=black}, scale=1]
		
		\node (A) at (-1.5,-1.5) {$\mathbb{A}$};
		\node (A delta) at (-1.5,1.5) {$\mathbb{A}^{\delta}$};
		\node (B) at (1.5,-1.5) {$\mathbb{B}$};
		\node (B delta) at (1.5,1.5) {$\mathbb{B}^{\delta}$};
		
		\draw [right hook->] (A) to  (A delta);
		\draw [right hook->] (B)  to (B delta);
		\draw [right hook->] (A)  to node[above] {$e$}  (B);
		\draw [right hook->] (A delta)  to node[above] {$e^{\delta}$}  (B delta);
		
		\draw [->] (B delta) to [out=135,in=45, looseness=1]  node[below] {\rotatebox[origin=c]{270}{$\vdash$}} node[above] {$\iota^{\pi}$}  (A delta);
		\draw [->] (B) to [out=135,in=45, looseness=1]  node[below] {\rotatebox[origin=c]{270}{$\vdash$}} node[above] {$\iota$}  (A);
		\draw [->] (B delta) to [out=225,in=-45, looseness=1]  node[above] {\rotatebox[origin=c]{270}{$\vdash$}} node[below] {$c^{\sigma}$}  (A delta);
		\draw [->] (B) to [out=225,in=-45, looseness=1]  node[above] {\rotatebox[origin=c]{270}{$\vdash$}} node[below] {$c$}  (A);
		\end{tikzpicture}
	\end{center}
\end{prop}
\begin{proof}
	By Proposition \ref{prop:existence of algebras and adjoints:biHeyting}, to complete the proof of the first part of the statement, we  need to address the claims regarding the expansions. This is done by using a   version of the  duality in \cite{sofronie2000duality} restricted to those Priestley spaces which are also bi-Esakia spaces. As to the second part, notice that the commutativity of the diagrams \eqref{eq: e and c and iota and f and g} can be written in the form of pairs of  inequalities (i.e.~$f = cf^\circ e^{\epsilon_f}$ and $g = \iota g^\circ e^{\epsilon_g}$) which are Sahlqvist and hence lift to the upper part of the diagram above.
\end{proof}

\section{Correspondence via translation}\label{sec:Corresp:Via:Trans}
The theory developed so far puts us in a position to meaningfully formulate and prove the transfer of first-order correspondence to Sahlqvist and inductive DLE-inequalities from suitable classical poly-modal cases. These general results specialize to the logics mentioned above, e.g. those mentioned in Example \ref{ex:various DLE-languages}.

In what follows, we let $\mathcal{L}$ denote an arbitrary but fixed DLE-language and $\mathcal{L}^\circ$ its associated  target language (cf.~Section \ref{sssection:parametric DLE}).  
The general definition of inductive inequalities (cf.\ Definition \ref{Inducive:Ineq:Def}) applies both $\mathcal{L}$ and $\mathcal{L}^\circ$. In particular, the Boolean negation in $\mathcal{L}^\circ$ enjoys both the order-theoretic properties of a unary $f$-type connective and of a unary $g$-type connective. Hence, Boolean negation occurs unrestricted in inductive $\mathcal{L}^\circ$-inequalities. Moreover, the algebraic interpretations of the  S4-connectives $\Box_\leq$ and $\Diamond_\geq$ enjoy the order-theoretic properties of normal unary $f$-type and $g$-type connectives respectively. Hence,  the occurrence of $\Box_\leq$ and $\Diamond_\geq$ in inductive $\mathcal{L}^\circ$-inequalities is subject to the same restrictions applied to any connective pertaining to the same class to which they belong.

The following correspondence theorem is a straightforward extension to the $\mathcal{L}^\circ$-setting of the correspondence result for classical normal modal logic in \cite{CoGoVa06}:
\begin{prop}\label{prop:correspondence theorem in lstar}
	Every inductive $\mathcal{L}^\circ$-inequality has a first-order correspondent over its class of $\mathcal{L}^\circ$-frames.
\end{prop}

In what follows, we aim to transfer the correspondence theorem for inductive $\mathcal{L}^{\circ}$-inequalities as stated in the proposition  above to inductive $\mathcal{L}$-inequalities. The next proposition is the first step towards this goal. As before, let $X := \mathsf{Prop}$.
\begin{prop}\label{prop:preservation of inductive structure}
	The following are equivalent for any order-type $\epsilon$ on $X$, and any  $\mathcal{L}$-inequality $\phi\leq \psi$:
	\begin{enumerate}
		\item $\phi\leq \psi$ is an $(\Omega, \epsilon)$-inductive $\mathcal{L}$-inequality;
		\item  $\tau_\epsilon(\phi)\leq \tau_\epsilon(\psi)$ is an $(\Omega, \epsilon)$-inductive $\mathcal{L}^\circ$-inequality.
	\end{enumerate}
\end{prop}
\begin{proof}
	By induction on the shape of $\phi\leq \psi$. In a nutshell: the definitions involved guarantee that: (1) PIA nodes are introduced immediately above $\epsilon$-critical occurrences of proposition variables;  (2) Skeleton nodes are translated as (one or more) Skeleton nodes; (3) PIA nodes are translated as (one or more) PIA nodes. Moreover, this translation does not disturb the dependency order $\Omega$. Hence, from item 1 to item 2, the translation does not introduce any violation on $\epsilon$-critical branches, and, from item 2 to item 1, the translation does not amend any violation.
\end{proof}

\begin{theorem}[Correspondence via translation]\label{theor:corr via transl}
	The correspondence theorem for inductive $\mathcal{L}^{\circ}$-inequalities transfers to inductive $\mathcal{L}$-inequalities.
\end{theorem}
\begin{proof}
	Let $\phi\leq \psi$ be an $(\Omega, \epsilon)$-inductive $\mathcal{L}$-inequality, and $\mathbb{F}$ be an $\mathcal{L}$-frame such that $\mathbb{F}\Vdash \phi\leq \psi$. By the discrete duality between perfect $\mathcal{L}$-algebras and $\mathcal{L}$-frames, this assumption is equivalent to $\bbA\models \phi\leq \psi$, where $\bbA$ denotes the complex $\mathcal{L}$-algebra of $\mathbb{F}$. Since $\bbA$ is a perfect $\mathcal{L}$-algebra, by Proposition  
	\ref{prop:existence of algebras and adjoints:perfectDLEs}, a perfect $\mathcal{L}^\circ$-algebra $\bbB$ exists with a natural embedding $e:\bbA\rightarrow\bbB$  which is a homomorphism of the lattice reducts of $\bbA$ and $\bbB$ and  has both a right adjoint $\iota\colon \bbB\to \bbA$ and a left adjoint $c\colon \bbB\to \bbA$ such that diagrams \eqref{eq: e and c and iota and f and g} commute. Hence, Corollary 
	\ref{cor:main theorem for tau epsilon DLE} is applicable and yields $\bbA\models \phi\leq \psi$ iff $\bbB\models \tau_\epsilon(\phi)\leq \tau_\epsilon(\psi)$, which is equivalent to $\mathbb{F}\Vdash^*\tau_\epsilon(\phi)\leq \tau_\epsilon(\psi)$  by the discrete duality between perfect $\mathcal{L}^\circ$-algebras and $\mathcal{L}^\circ$-frames.
	
	By Proposition \ref{prop:preservation of inductive structure}, $\tau_\epsilon(\phi)\leq \tau_\epsilon(\psi)$ is an $(\Omega, \epsilon)$-inductive $\mathcal{L}^\circ$-inequality, and hence, by Proposition \ref{prop:correspondence theorem in lstar}, $\tau_\epsilon(\phi)\leq \tau_\epsilon(\psi)$ has a first-order correspondent $\mathsf{FO}(\phi)$ on $\mathcal{L}^\circ$-frames. Therefore $\mathbb{F}\Vdash^* \tau_\epsilon(\phi)\leq \tau_\epsilon(\psi)$ iff $\mathbb{F}\models\mathsf{FO}(\phi)$. Since the first-order frame correspondence languages of $\mathcal{L}$ and $\mathcal{L}^\circ$ are the same, it follows that $\mathsf{FO}(\phi)$ is also the first-order correspondent of $\phi\leq \psi$.
	The steps of this argument are summarized in  the following chain of equivalences:
  \[
    \text{
		\begin{tabular}[b]{r l r}
			&$\mathbb{F}\Vdash\phi\leq \psi$&\\
			iff &$\bbA\models\phi\leq \psi$& (discrete duality for $\mathcal{L}$-frames)\\
			iff &$\bbB \models\tau_{\epsilon}(\phi)\leq \tau_{\epsilon}(\psi)$&(Proposition \ref{prop:existence of algebras and adjoints:perfectDLEs}, Corollary 
			\ref{cor:main theorem for tau epsilon DLE})\\
			iff &$\mathbb{F}\Vdash^*\tau_{\epsilon}(\phi)\leq \tau_{\epsilon}(\psi)$&(discrete duality for $\mathcal{L}^\circ$-frames)\\
			iff &$\mathbb{F}\models\mathsf{FO}(\phi)$&(Proposition \ref{prop:correspondence theorem in lstar})
		\end{tabular}}
    \tag*{\qedhere}
  \]
\end{proof}

\begin{remark}
	In Example \ref{ex:various DLE-languages}, we showed that the languages of  Rauszer's bi-intuitionistic logic, Fischer Servi's intuitionistic modal logic, Wolter's bi-intuitionistic modal logic, Bezhanishvili's MIPC with universal modalities, Dunn's positive modal logic and Gehrke Nagahashi and Venema's Distributive Modal Logic are specific instances of DLE-logics. Hence, Theorem \ref{theor:corr via transl}, applied to each of these settings, enables one to transfer generalized Sahlqvist correspondence theorems to each of these logics. In all these settings but the latter two, this transferability is a new result which can be added to the list of known transfer results for these logics (cf.~e.g.~\cite[Section 4.1]{chagrov1992modal}). In positive modal logic and distributive modal logic, \cite[Theorem 3.7]{GeNaVe05} proves the transfer of Sahlqvist correspondence. Our result strengthens this to the transfer of correspondence for the larger class of inductive formulas.\footnote{In \cite{GorankoV06} it is shown that inductive formulas exist which are not semantically equivalent to any Sahlqvist formula.}
\end{remark}

\section{Canonicity via translation}\label{sec:canonicity}

In this section we apply the results of Section \ref{sec:instantiations}, and in particular those of Section \ref{sec:ParamGMTbHAE}, to show that the canonicity of the inductive $\mathcal{L}_{\mathrm{bHAE}}$-inequalities transfers from classical multi-modal logic via parametrized GMT translations. The proof strategy of this result does not generalize successfully to DLE-logics or, indeed, to intuitionistic or co-intuitionistic modal logics. We discuss the reasons for this and propose possible alternative strategies.


\subsection{Canonicity transfer to inductive bHAE-inequalities}\label{ssec:canonicity bHAE}

Throughout the present section, let us fix a bHAE-signature $\mathcal{L}_{\mathrm{bHAE}} = \mathcal{L}_{\mathrm{bHAE}}(\mathcal{F}, \mathcal{G})$, and let   $\mathcal{L}^\circ_{\mathrm{BAE}} =  \mathcal{L}_{\mathrm{BAE}}(\mathcal{F}^\circ, \mathcal{G}^\circ)$ be the target language for the parametric GMT translations for $\mathcal{L}_{\mathrm{bHAE}}$ (cf.~Section \ref{sssection:parametric DLE}). The following canonicity theorem is a straightforward algebraic reformulation of the canonicity result for classical normal polyadic modal logic in \cite{GorankoUnleashed} and \cite{ConGorVakSQEMA2}:
\begin{prop}\label{prop:canonicity theorem in lstar}
	For every inductive $\mathcal{L}^\circ_{\mathrm{BAE}}$-inequality $\alpha\leq\beta$ and every $\mathcal{L}^\circ_{\mathrm{BAE}}$-algebra $\bbB$,
	\[\mbox{if } \bbB\models \alpha\leq\beta \ \mbox{ then } \bbB^\delta\models \alpha\leq\beta.\]
\end{prop}

In what follows, we show that the canonicity of inductive $\mathcal{L}^\circ_{\mathrm{BAE}}$-inequalities, given by the proposition above, transfers to inductive $\mathcal{L}_{\mathrm{bHAE}}$-inequalities via suitable parametrized GMT translations.


\begin{theorem}[Canonicity  via translation]\label{theor:canon via transl bi HAE}
	The canonicity theorem for inductive $\mathcal{L}^\circ_{\mathrm{BAE}}$-inequalities transfers to inductive  $\mathcal{L}_{\mathrm{bHAE}}$-inequalities.
\end{theorem}
\begin{proof}
	Fix an $\mathcal{L}_{\mathrm{bHAE}}$-algebra $\bbA$ and let $\phi\leq\psi$ be an inductive $\mathcal{L}_{\mathrm{bHAE}}$-inequality such that $\bbA\models \phi\leq \psi$. We show $\bba^\delta\models \phi\leq \psi$. By Proposition \ref{prop:existence of algebras and adjoints:bHAE},  an $\mathcal{L}^\circ_{\mathrm{BAE}}$-algebra $\bbB$ exists with a natural embedding $e:\bbA\hookrightarrow\bbB$  which is a homomorphism of the lattice reducts of $\bbA$ and $\bbB$ and  has both a right adjoint $\iota\colon \bbB\to \bbA$ and a left adjoint $c\colon \bbB\to \bbA$ such that conditions \eqref{eq: rightarrow Rightarrow e and iota}, \eqref{eq: pdra minus e and c}, and \eqref{eq: e and c and iota and f and g} hold. Hence, Corollary \ref{cor:main theorem for tau epsilon bi} is applicable, which yields $\bbA\models \phi\leq \psi$ iff $\bbB\models \tau_\epsilon(\phi)\leq \tau_\epsilon(\psi)$.
	
	By Proposition \ref{prop:preservation of inductive structure}, $\tau_\epsilon(\phi)\leq \tau_\epsilon(\psi)$ is an $(\Omega, \epsilon)$-inductive $\mathcal{L}^\circ_{\mathrm{BAE}}$-inequality, and hence, by Proposition \ref{prop:canonicity theorem in lstar}, $\bbB^\delta\models\tau_\epsilon(\phi)\leq \tau_\epsilon(\psi)$. By the last part of the statement of Proposition \ref{prop:existence of algebras and adjoints:bHAE}, Corollary \ref{cor:main theorem for tau epsilon bi} applies also to $\bbA^\delta$ and $\bbB^\delta$, and thus $\bbA^\delta\models \phi\leq \psi$, as required.
	The steps of this argument are summarized in  the following U-shaped diagram:
  \[
		\begin{tabular}[b]{l c l}\label{table:U:shape:algebra}
			
			$\bbA\models\phi\leq\psi$ & & $\bbad\models \phi\leq\psi$ \\
			$\ \ \ \ \Updownarrow$ (Prop \ref{prop:existence of algebras and adjoints:bHAE}, Cor \ref{cor:main theorem for tau epsilon bi})& & $\ \ \ \ \ \Updownarrow $ (Prop \ref{prop:existence of algebras and adjoints:bHAE}, Cor \ref{cor:main theorem for tau epsilon bi})\\
			$\bbB\models\tau_\epsilon(\phi)\leq\tau_\epsilon(\psi)$ &\ \ \ $\Leftrightarrow$ \ \ \  &  $\bbbd\models\tau_\epsilon(\phi)\leq\tau_\epsilon(\psi)$\\
		\end{tabular}
    \tag*{\qedhere}
  \]
\end{proof}

\begin{remark}
	\label{rmk: comparison canonicity}
	Theorem \ref{theor:canon via transl bi HAE} applies to both  Rauszer's bi-intuitionistic logic\footnote{As discussed in Example \ref{ex: non inductive}, not all axioms in Rauszer's axiomatization of bi-intuitionistic logic are inductive. However, in \cite{gore2000dual}, Gor\'e introduces a proper display calculus for bi-intuitionistic logic, which implies, by the characterization given in \cite{GrMaPaTzZh15}, that an axiomatization for bi-intuitionistic logic exists which consists only of inductive formulas.}  \cite{rauszer1974semi} and  Wolter's bi-intuitionistic modal logic \cite{wolter1998CoImplication}. Hence,   transfer of generalized Sahlqvist canonicity theorems is available for each of these logics.  This transferability is a new result for these settings, and is different  from the transfer of d-persistence as was shown e.g.~in \cite[Theorem 12]{WoZa98} in the context of intuitionistic modal logic, in at least two respects; first, it hinges specifically on the preservation and reflection of the shape of inductive formulas; second, it does not rely on any assumptions about the interaction between the S4 and other modalities in the target logic such as those captured by the mix axiom (cf.~Remark \ref{rmk: mix sufficient but not necessary}).
\end{remark}

\subsection{Generalizing the canonicity-via-translation argument}\label{ssec:Gen:Canon:Does:Not:Work}

In the present subsection, we discuss the extent to which the  proof pattern described in the previous subsection can be applied  to the settings of normal Heyting and co-Heyting algebra expansions (HAEs, cHAEs), and to normal DLEs. In the setting of  bHAEs, the order embedding $e$ has both a left and a right adjoint, the existence of which is shown in Proposition \ref{prop:existence of algebras and adjoints:bHAE}, while for HAEs, cHAEs and DLEs 
at most one of the two adjoints was shown to exist  in general (cf.\ Propositions \ref{prop:existence of algebras and adjoints:Heyting} and  \ref{prop:existence of algebras and adjoints:coHeyting}), while both adjoints exist if the algebra is perfect (cf.~Proposition \ref{prop:existence of algebras and adjoints:perfectDLEs}). 

This implies that the  the U-shaped argument discussed in the proof of Theorem \ref{theor:canon via transl bi HAE} 
is not straightforwardly applicable to HAEs, cHAEs and DLEs. Indeed, in each of these settings, the  equivalence  on the side of the perfect algebras can still be argued using Proposition \ref{prop:existence of algebras and adjoints:perfectDLEs} and Corollary \ref{cor:main theorem for tau epsilon DLE}, 
but the one on the side of general algebras  (left-hand side of the diagram) cannot, precisely because Proposition \ref{prop:existence of algebras and adjoints:perfectDLEs} does not generalize to arbitrary DLEs (resp.~HAEs, cHAEs).

\begin{center}
	\begin{tabular}{l c l}
		
		$\bbA\models\phi\leq\psi$ & & $\bbad\models \phi\leq\psi$ \\
		$\ \ \ \ \Updownarrow$ ?  & & $\ \ \ \ \ \Updownarrow $ (Prop \ref{prop:existence of algebras and adjoints:perfectDLEs}, Cor \ref{cor:main theorem for tau epsilon DLE})\\
		$\bbB\models\tau_\epsilon(\phi)\leq\tau_\epsilon(\psi)$ &\ \ \ $\Leftrightarrow$ \ \ \  &  $\bbbd\models\tau_\epsilon(\phi)\leq\tau_\epsilon(\psi)$\\
	\end{tabular}
\end{center}
In what follows, we employ a more refined argument to show that the left-hand side equivalence holds. That is, the question mark in the U-shaped diagram above can be replaced by Proposition \ref{prop: main prop of Godel tarski topologically} below. We work in the setting of $\mathcal{L}$-algebras for an arbitrarily fixed DLE-signature $\mathcal{L}$, with $\mathcal{L}^\circ$ its associated target signature. 
Recall that the canonical extension $e^\delta: \bbad\to \bbbd$ of the embedding $e: \bbA\hookrightarrow \bbB$ is a complete lattice homomorphism, and hence both its left and right adjoints exist, which we respectively denote $c: \bbbd\to \bbad$ and $\iota: \bbbd\to \bbad$. 
It is well known from the theory of canonical extensions that  $c(b)\in K(\bbad)$ and $\iota(b)\in O(\bbad)$ for every $b\in \bbB$ (cf.~\cite[Lemma 10.3]{ConPal13}). 
Hence, if $r_\epsilon: (\bbbd)^X\to (\bbad)^X$ is the map defined for any $U\in (\bbbd)^X$ and $p\in X$ by:
\[
r_\epsilon (U)(p) = \begin{cases}
(\iota\circ U) (p) & \mbox{if } \epsilon(p) = 1\\
(c\circ U) (p) & \mbox{if } \epsilon(p) = \partial\\
\end{cases}
\]

\noindent then  $(r_{\epsilon}(U)) (p) \in \kbbas $ if $\epsilon(p) = \partial$ and $ (r_{\epsilon}(U))(p) \in \obbas$ if $\epsilon(p) = 1$
for any `admissible valuation' $U\in \bbB^X$ and $p\in X$. 

\begin{lemma}\label{lemma:(a) and (b) hold refined version}
	Let  $\bbA$ be an $\mathcal{L}$-algebra, and $e:\bbA\hookrightarrow\bbB$ be an embedding of $\bbA$ into an $\mathcal{L}^\circ$-algebra $\bbB$ which is a homomorphism of the lattice reducts of $\bbA$ and $\bbB$  such that the left and right adjoints of $e^\delta: \bbad\to \bbbd$ make the diagrams \eqref{eq: e and c and iota and f and g} commute.
	Then, for every order-type $\epsilon$ on $X$,  
	the following conditions hold for every $\phi \in \mathcal{L}$:
	\begin{enumerate}[label={(\alph*)}]
		\item $e(\val{ \phi}_V) = \valb{\tau_\epsilon(\phi)}_{\overline{e}(V)}$ for every $V \in \bbA^X$;
		\item $\valb{ \tau_\epsilon (\phi) }_U = e^\delta(\val{\phi}_{r_\epsilon(U)})$ for every $U \in \bbB^X$.
	\end{enumerate}
\end{lemma}
\begin{proof}
	The statement immediately follows from Proposition \ref{prop:  tau epsilon for DLE satisfies a and b} applied to  $e^\delta: \bbad\to \bbbd$. 
\end{proof}

\begin{prop}\label{prop: main prop of Godel tarski topologically}
	Let  $\bbA$ be an $\mathcal{L}$-algebra, and $e:\bbA\hookrightarrow\bbB$ be an embedding of $\bbA$ into an $\mathcal{L}^\circ$-algebra $\bbB$ which is a homomorphism of the lattice reducts of $\bbA$ and $\bbB$ such that the left and right adjoints of $e^\delta: \bbad\to \bbbd$ make the diagrams \eqref{eq: e and c and iota and f and g} commute.  
	Then, 
	for every $(\Omega, \epsilon)$-inductive $\mathcal{L}$-inequality $\phi \leq \psi$,
	\[
	\bbA^\delta \models_\bbA \phi\leq \psi \quad \mbox{ iff } \quad \bbB^\delta \models_\bbB \tau_{\epsilon}(\phi)\leq \tau_{\epsilon}(\psi).
	\]
\end{prop}
\begin{proof}[Sketch of proof]
	From right to left, if $(\bbA^\delta,V) \not\models \phi\leq \psi$ for some $V \in \bbA^X$, then $\val{\phi}_{V}\not\leq \val{\psi}_{V}$. By Lemma \ref{lemma:(a) and (b) hold refined version} (a), 
	this implies that $\valb{\tau_{\epsilon}(\phi)}_{\overline{e}(V)} = e(\val{\phi}_{V})\not\leq e(\val{\psi}_{V}) = \valb{\tau_{\epsilon}(\psi)}_{\overline{e}(V)}$, that is $(\bbB^\delta, \overline{e}(V)) \not\models \tau_{\epsilon}(\phi)\leq \tau_{\epsilon}(\psi)$, as required.
	
	Conversely,  assume contrapositively that $(\bbB^\delta,U) \not\models \tau_{\epsilon}(\phi)\leq \tau_{\epsilon}(\psi)$ for some $U \in \bbB^X$, that is, $\valb{\tau_{\epsilon}(\phi)}_U\not\leq \valb{\tau_{\epsilon}(\psi)}_U$. By Lemma \ref{lemma:(a) and (b) hold refined version}  (b), this  is equivalent to $e^\delta(\val{\phi}_{r_\epsilon(U)})\not\leq e^\delta(\val{\psi}_{r_\epsilon(U)})$, which, by the monotonicity of $e^\delta$, implies that $\val{\phi}_{r_{\epsilon}(U)}\not\leq \val{\psi}_{r_{\epsilon}(U)}$, that is,  $(\bbA, r_{\epsilon}(U)) \not\models \phi\leq \psi$. This is not enough to finish the proof, since $r_{\epsilon}(U)$ is not guaranteed to belong in $\bbA^X$; however, as observed above, $(r_{\epsilon}(U)) (p) \in \kbbas $ if $\epsilon(p) = \partial$ and $ (r_{\epsilon}(U))(p) \in \obbas$ if $\epsilon(p) = 1$ for each proposition variable $p$. To finish the proof, we need to show that an admissible valuation $V' \in\bbA^X$ can be manufactured from $r_{\epsilon}(U)$ and $\phi\leq\psi$ in such a way that $(\bbA^\delta, V') \not\models\phi \leq \psi$.
	In what follows, we provide a sketch of the proof of the existence of the required $V'$.
	Assume that $\epsilon(q)=\partial$  for some proposition variable $q$  occurring in $\phi\leq\psi$ (the case of $\epsilon(q)=1$ is analogous and is omitted). Then
	we define $V'(q)\in\bbA$ 
	as follows. We run ALBA on $\phi\leq\psi$ according to the dependency order $<_\Omega$, up to the point when we solve for the negative occurrences of $q$, which by assumption are $\epsilon$-critical. Notice that ALBA preserves truth under assignments.\footnote{In \cite{ConPal12} it is proved that ALBA steps preserve validity of quasi-inequalities. In fact,  something stronger is ensured, namely that truth under assignments is preserved, modulo the values of introduced and eliminated variables. This notion of equivalence is studied in e.g.\ \cite{conradie:SQEMA4}. We are therefore justified in our assumption that the value of $q$ is held constant as are the values of all variables occurring in $\phi \leq \psi$ which have not yet been eliminated up to the point where $q$ is solved for.} Then the inequality providing the minimal valuation of $q$ is of the form $q\leq\alpha$, where $\alpha$ is {\em pure} (i.e.\ no proposition variables occur in $\alpha$). By  \cite[Lemma 9.5]{ConPal12}, every inequality in the antecedent of the quasi-inequality obtained by applying first approximation to an inductive inequality is of the form $\gamma\leq\delta$ with $\gamma$ syntactically closed and $\delta$ syntactically open. Hence, $\alpha$ is pure and syntactically open, which means that the interpretation of $\alpha$ is an element in $\obbas$. Therefore, by compactness, there exists some $a\in\bbA$ such that $r_{\epsilon}(U)(q)\leq a\leq\alpha$. Then we define $V'(q)=a$. Finally, it remains to be shown that $(\bbA^\delta, V')\not\models\phi \leq \psi$. This immediately follows from the fact that ALBA steps preserves truth under assignments, and that all the inequalities in the system are preserved in the change from  $r_{\epsilon}(U)$ to $V'$.
\end{proof}

However,  Proposition \ref{prop: main prop of Godel tarski topologically} is still not enough for the U-shaped argument above to go through. Indeed, 
notice that, whenever $e: \bbA\to \bbB$ misses one of the two adjoints (e.g.~the left adjoint),  for any $(\Omega, \epsilon)$-inductive $\mathcal{L}$-inequality $\phi\leq\psi$ containing some $q$ with $\epsilon(q)=\partial$, its translation $\tau_{\epsilon}(\phi)\leq\tau_{\epsilon}(\psi)$ contains occurrences of the connective $\Diamond_{\geq}$, the algebraic interpretation of which in $\bbbd$ is based on the left adjoint $c$ of $e^\delta$, which, as discussed above, maps elements in $\bbB$ to elements in $K(\bbbd)$. Hence, the canonicity of $\tau_{\epsilon}(\phi)\leq\tau_{\epsilon}(\psi)$, understood as  the preservation of its validity from $\bbB$ to $\bbB^\delta$, cannot be argued by appealing to  Proposition \ref{prop:canonicity theorem in lstar}: indeed, Proposition \ref{prop:canonicity theorem in lstar}  holds under the assumption that $\bbB$ is an $\mathcal{L}^\circ$-subalgebra of $\bbbd$, while, as discussed above, $\bbB$ is not in general closed under $\Diamond_{\geq}$.

In order to be able to adapt the canonicity-via-translation argument to the case of HAEs, cHAEs and DLEs, we would need to  strengthen Proposition \ref{prop:canonicity theorem in lstar} so as  to obtain 
the following equivalence for any 
inductive $\mathcal{L}^\circ$-inequality $\alpha\leq \beta$:
\begin{equation}\label{eq:generalized canonicity}
\bbB^\delta\models_{\bbB} \alpha\leq\beta\ \mbox{ iff }\ \bbB^\delta\models \alpha\leq\beta\end{equation} in a setting in which the interpretations of $\Diamond_{\geq}$ and $\Box_{\leq}$ exist only in $\bbB^{\delta}$ and all we have in general for $b \in \bbB$ is that $\Diamond^{\bbB^{\delta}}(b)\in K(\bbB^\delta)$  and $\Box^{\bbB^{\delta}}(b)\in O(\bbB^\delta)$.

Such a  strengthening cannot be straightforwardly obtained with the tools provided by the present state-of-the-art in canonicity theory.
To see where the problem lies, let us try and apply  ALBA/SQEMA in an attempt to prove the left-to-right direction of \eqref{eq:generalized canonicity} for the `Sahlqvist' inequality $\Box_{\leq}  p\leq  \Diamond_\geq \Box_{\leq} p$, assuming that $\Diamond_\geq$ is left adjoint to $\Box_\leq$, and  $\valb{\Box_{\leq}  p}_U\in O(\bbbd)$ and $\valb{\Diamond_\geq  p}_U\in K(\bbbd)$ for any admissible valuation $U\in \bbB^X$:
\begin{center}
	\begin{tabular}{cll}
		& $\bbbd\models_\bbB\forall p[\Box_{\leq}  p\leq  \Diamond_\geq \Box_{\leq} p]$\\
		iff & $\bbbd\models_\bbB\forall p\forall \nomi\forall\cnomm[(\nomi\leq  \Box_{\leq}  p \ \&\ \Diamond_\geq \Box_{\leq}  p\leq \cnomm)\Rightarrow \nomi\leq\cnomm]$\\
		iff & $\bbbd\models_\bbB\forall p\forall \nomi\forall\cnomm[(\Diamond_\geq\nomi\leq   p \ \&\ \Diamond_\geq \Box_{\leq}  p\leq \cnomm)\Rightarrow \nomi\leq\cnomm]$\\
	\end{tabular}
\end{center}
The minimal valuation term  $\Diamond_\geq\nomj$, computed by ALBA/SQEMA when solving for the negative occurrence of $p$, is closed. However, substituting this minimal valuation  into $\Diamond_\geq \Box_{\leq}  p\leq \cnomm$ would get us $\Diamond_\geq \Box_{\leq}   \Diamond_\geq\nomj\leq \cnomm$ with $\Diamond_\geq \Box_{\leq}   \Diamond_\geq\nomj$  neither closed nor open.
Hence, we cannot anymore appeal to the Esakia lemma in order to prove the following equivalence:\footnote{In other words, if $\bbB$ is not closed under $\Diamond_\geq$ or $\Box_\leq$,  the soundness of the application of the Ackermann rule under admissible assignments cannot be argued  anymore by appealing to the Esakia lemma, and hence, to the topological Ackermann lemma.}
\begin{center}
	\begin{tabular}{cll}
		& $\bbbd\models_\bbB\forall p\forall \nomi\forall\cnomm[(\Diamond_\geq\nomi\leq   p \ \&\ \Diamond_\geq \Box_{\leq}  p\leq \cnomm)\Rightarrow \nomi\leq\cnomm]$\\
		iff & $\bbbd\models_\bbB \forall \nomi\forall\cnomm[\Diamond_\geq \Box_{\leq}\Diamond_\geq\nomi\leq \cnomm\Rightarrow \nomi\leq\cnomm]$\\
	\end{tabular}
\end{center}

An analogous situation arises when solving for the positive occurrence of $p$. Other techniques for proving canonicity, such as J\'onsson-style canonicity \cite{Jonsson94,PaSoZh15a}, display the same problem, since they also rely on an Esakia lemma which is not available if $\bbB$ is not closed under $\Box_{\leq}$ and $\Diamond_\geq$.

\section{Conclusions and further directions}\label{Sec:Conclusions}

\subsection*{Contributions.} In the present paper, we have laid the groundwork for a general and uniform theory of  transfer of generalized Sahlqvist correspondence and canonicity from normal BAE-logics to normal DLE-logics. Towards this goal, we have introduced  a unifying template for GMT translations, of which the GMT translations in the literature can be recognized as instances. We have proved that generalized Sahlqvist correspondence transfers for all DLE-logics, while generalized Sahlqvist canonicity transfers for the more restricted setting of bHAE-logics. The formulation of these results has been made possible by the recent introduction of a general mechanism for identifying Sahlqvist and inductive classes for any normal DLE-signature \cite{ConPal13}. Consequently, there are not many transfer results in the literature to which these results can be compared, the only exceptions being the transfer of Sahlqvist correspondence for DML-inequalities of \cite[Theorem 3.7]{GeNaVe05}, and  the transfer of canonicity in the form of d-persistence for intuitionistic modal formulas of \cite[Theorem 12]{WoZa98}. The transfer of correspondence shown in this paper generalizes \cite[Theorem 3.7]{GeNaVe05} both as regards the setting (from DML to general normal DLE-logics) and the scope (from Sahlqvist to inductive inequalities). As discussed in Remark \ref{rmk: comparison canonicity}, the transfer of canonicity shown in the present paper is neither subsumed by, nor  does it subsume  \cite[Theorem 12]{WoZa98}. 

Regarding insights, we have also gained a better understanding of the nature  of the difficulties in the transfer of (generalized) Sahlqvist canonicity via GMT translations. These difficulties were discussed as follows in the conclusions  of \cite{GeNaVe05}: \begin{quote}[...] a reduction to
	the classical result for canonicity seems to be much harder than for correspondence, due
	to the following reason. In the correspondence case, where we are working with perfect
	DMAs, there is an obvious way to connect with Boolean algebras with operators, namely
	by taking the (Boolean) complex algebra of the dual frame. In the canonicity case however,
	we would need to embed arbitrary DMAs into BAOs in a way that would interact nicely with
	taking canonical extensions, and we do not see a natural, general way for doing so.
\end{quote}
Our analysis shows that, actually,  the problem does not lie in the interaction between the embedding and the canonical extensions,  but rather in the fact that the embedding $e: \bbA\hookrightarrow  \bbB$ of an arbitrary DLE into a suitable BAE lacks the required adjoint maps, and that while the role of these adjoints can be played to a certain extent by the adjoints of $e^\delta: \bbA^\delta\hookrightarrow  \bbB^\delta$ (cf.~Proposition \ref{prop: main prop of Godel tarski topologically}), we would need to develop a much stronger theory of algebraic (generalized) Sahlqvist canonicity in the BAE setting to be able to reduce (generalized) Sahlqvist canonicity for DLEs to the Boolean setting.

\subsection*{Further directions.} As  mentioned above,  our analysis suggests a way  to obtain the transfer of generalized Sahlqvist canonicity for arbitrary DLE-logics, namely   to develop a generalized canonicity theory in the setting of BAEs which relies on the order-theoretic properties of maps $\bbA\to \bbB^\delta$. The way to this theory has already been paved in \cite{PaSoZh15a}, where  a generalization of the standard theory of  canonical extensions of maps is developed, accounting for maps $f^{\bbA}:\bba\to\bbB^{\delta}$   such that the value of $f^{\bbA}$ is not restricted to clopen elements in $\bbB$.
\subsection*{Blok-Esakia theorem for DLE-logics.} The uniform perspective on GMT translations developed in this paper can perhaps be useful to systematically explore the possible variants of the notion of `modal companion' of a given intuitionistic modal logic, and extend the Blok-Esakia theorem uniformly to DLE-logics.
\subsection*{Methodology: generalization through algebras via duality.} The generalized canonicity-via-translation result for the bi-intuitionistic setting comes from embracing the full extent of the algebraic analysis. Specifically, canonicity-via-translation hinges upon the fact that the  interplay of persistent and non-persistent valuations on frames can be understood and reformulated in terms of an adjunction situation between two complex algebras of the same frame.  In its turn, this adjunction situation generalizes to arbitrary algebras. The same modus operandi, which achieves generalization through algebras via duality, has been fruitfully employed by some of the authors also for very different purposes, such as the definition of the non-classical counterpart of a given logical framework (cf.\  \cite{KurzPalmiLMCS2013, MaPaSa14, CoFrPaTz15}).

\section*{Acknowledgement}
The authors wish to thank the anonymous reviewer for their suggestions which led to substantial improvements of the paper.

\bibliographystyle{abbrv}
\bibliography{translation}
\end{document}